\documentclass [12pt]{amsart}
\usepackage[utf8]{inputenc}
\pdfoutput=1

\usepackage{amsmath,amssymb,amsthm,amsfonts}
\usepackage{mathtools}%

\usepackage[main=english,french]{babel}%
\usepackage[bookmarksdepth=4]{hyperref}%
\usepackage{bbm}%
\usepackage{mathrsfs}%

\usepackage{tikz,graphicx,color}
\usepackage{tikz-cd}%
\usepackage{tikz-3dplot}%
\usetikzlibrary{calc}%
\usetikzlibrary{arrows}%
\usetikzlibrary{shapes}%
\usetikzlibrary{patterns}%
\usetikzlibrary{positioning}%
\usetikzlibrary{arrows.meta}
\usetikzlibrary{decorations.markings}
\usepackage{epstopdf}%
\pdfpageattr{/Group <</S /Transparency /I true /CS /DeviceRGB>>}

\usepackage[arrow]{xy}%
\usepackage{diagbox}%
\usepackage{subfig}%
\usepackage{arcs}%
\usepackage{xcolor}%
\usepackage{tabu}
\usepackage{booktabs}%

\usepackage[margin=1in]{geometry}

\usepackage{soul}
\usepackage{accents}
\usepackage{enumitem}
\usepackage{letltxmacro}
\usepackage{thmtools,etoolbox}

\def\myarabic#1{\normalfont(\roman{#1})}
\newlist{theoremlist}{enumerate}{1}
\setlist[theoremlist]{label=\myarabic{theoremlisti},ref={\myarabic{theoremlisti}},itemindent=0pt,labelindent=0pt,
leftmargin=*,noitemsep}

\makeatletter
\renewcommand{\p@theoremlisti}{\perh@ps{\thetheorem}}
\protected\def\perh@ps#1#2{\textup{#1#2}}
\newcommand{\itemrefperh@ps}[2]{\textup{#2}}
\newcommand{\itemref}[1]{\begingroup\let\perh@ps\itemrefperh@ps\ref{#1}\endgroup}
\makeatother

\usepackage{nameref,hyperref}
\usepackage[capitalize]{cleveref}

\usepackage{tikz-cd}
\usepackage{stmaryrd}

\usepackage{xspace}

\usepackage{stmaryrd}

\newtheorem{theorem}{Theorem}[section]
\newtheorem*{statement}{Theorem}
\newtheorem{lemma}[theorem]{Lemma}
\newtheorem{proposition}[theorem]{Proposition}
\newtheorem{corollary}[theorem]{Corollary}
\newtheorem{conjecture}[theorem]{Conjecture}
\theoremstyle{definition}
\newtheorem{remark}[theorem]{Remark}
\theoremstyle{definition}
\newtheorem{definition}[theorem]{Definition}
\newtheorem{question}[theorem]{Question}
\theoremstyle{definition}

\theoremstyle{definition}
\newtheorem{example}[theorem]{Example}
\theoremstyle{definition}
\newtheorem{notation}[theorem]{Notation}

\crefname{figure}{Figure}{Figures}

\addtotheorempostheadhook[theorem]{\crefalias{theoremlisti}{theorem}}
\addtotheorempostheadhook[lemma]{\crefalias{theoremlisti}{lemma}}
\addtotheorempostheadhook[proposition]{\crefalias{theoremlisti}{proposition}}
\addtotheorempostheadhook[corollary]{\crefalias{theoremlisti}{corollary}}

\def\Acal{\mathcal{A}}\def\Fcal{\mathcal{F}}

\def\gbf{\mathbf{g}}\def\kbf{\mathbf{k}}

\def\C{\mathbb{C}}
\def\R{\mathbb{R}}

\def\Z{\mathbb{Z}}
\def\Q{\mathbb{Q}}
\def\P{\mathbb{P}}

\newcommand\parr[1]{{({#1})}}
\def\<{{\langle}}
\def\>{{\rangle}}

\def\la{{\lambda}}

\def\GL{\operatorname{GL}}
\def\SL{\operatorname{SL}}

\def\Gr{\operatorname{Gr}}

\def\Z{{\mathbb Z}}
\def\R{{\mathbb R}}
\def\Gr{{\rm Gr}}

\def\GL{{\rm GL}}

\def\id{{\operatorname{id}}}

\newcommand{\smat}[1]{\left[\begin{smallmatrix}
      #1
    \end{smallmatrix}\right]}

\def\pj{^{\parr J}}

\def\bs{\backslash}

\def\t{{\mathbf{t}}}

\def\d#1{\dot{#1}}
\def\ds{\d{s}}
\def\dw{\d{w}}
\def\du{\d{u}}
\def\dv{\d{v}}

\def\alphacheck{\alpha^\vee}

\makeatletter
\DeclareRobustCommand{\cev}[1]{%
  \mathpalette\do@cev{#1}%
}
\newcommand{\do@cev}[2]{%
  \fix@cev{#1}{+}%
  \reflectbox{$\m@th#1\vec{\reflectbox{$\fix@cev{#1}{-}\m@th#1#2\fix@cev{#1}{+}$}}$}%
  \fix@cev{#1}{-}%
}
\newcommand{\fix@cev}[2]{%
  \ifx#1\displaystyle
    \mkern#23mu
  \else
    \ifx#1\textstyle
      \mkern#23mu
    \else
      \ifx#1\scriptstyle
        \mkern#22mu
      \else
        \mkern#22mu
      \fi
    \fi
  \fi
}

\makeatother

\def\Rich_#1^#2{\vec R^\circ_{#1,#2}}
\def\LRich_#1^#2{\cev R^\circ_{#1,#2}}
\def\RichL_#1^#2{\LRich_{#1}^{#2}}
\def\Rtp_#1^#2{\vec R_{#1,#2}^{>0}}
\def\LRtp_#1^#2{\cev R_{#1,#2}^{>0}}

\def\Rtnn_#1^#2{\vec R_{#1,#2}^{\geq0}}
\def\LRtnn_#1^#2{\cev R_{#1,#2}^{\geq0}}

\def\Gtnn{{G_{\geq0}}}
\def\GBtnn{{(G/B_-)_{\geq0}}}
\def\GBptnn{{(G/B)_{\geq0}}}

\def\PR_#1^#2{\accentset{\circ}{\Pi}_{#1,#2}}%
\def\PRtp_#1^#2{\Pi_{#1,#2}^{>0}}%
\def\PRtnn_#1^#2{\Pi_{#1,#2}^{\geq0}}%
\def\PRcl_#1^#2{\Pi_{#1,#2}}%
\def\PRR_#1^#2{\accentset{\circ}{\Pi}_{#1,#2}^\R}%
\def\PRRcl_#1^#2{\Pi_{#1,#2}^\R}%

\def\bt{{\mathbf{t}}}
\def\bw{{\mathbf{w}}}
\def\bv{{\mathbf{v}}}

\def\Fcal{\mathcal{F}}

\def\pu#1{^{(#1)}}

\def\bx{{\mathbf{x}}}

\def\bi{{\mathbf{i}}}

\def\Gomp{G_0^\mp}
\def\Gopm{G_0^\pm}

\def\k{\kbf}

\def\Red{\operatorname{Red}}

\def\Hom{\operatorname{Hom}}

\def\hjmap{\kappa}
\def\hjmp_#1{\hjmap_{#1}}

\def\Uom_#1{U^{\diamond,-}_{#1}}

\def\xrasim{\xrightarrow{\sim}}
\def\Lxrasim{\xleftarrow{\sim}}

\def\Richaff_#1^#2{\accentset{\circ}{\mathcal{R}}_{#1}^{#2}}

\def\bs{\backslash}

\def\Cast{\C^\ast}

\def\Povar_#1{\accentset{\circ}{\Pi}_{#1}}
\def\Povarcl_#1{\Pi_{#1}}
\def\RPovar_#1{\accentset{\circ}{\Pi}^\R_{#1}}
\def\RPovarcl_#1{\Pi^\R_{#1}}
\def\Povtp_#1^#2{\Pi_{#1,#2}^{>0}}
\def\Povtnn_#1{\Pi_{#1}^{\geq0}}

\def\Star_#1{\operatorname{Star}_{#1}}
\def\Startnn_#1{\operatorname{Star}^{\geq0}_{#1}}

\def\Link{\operatorname{Lk}}
\def\Lkx_#1{\Link_{#1}}
\def\Lkxx_#1^#2{\accentset{\circ}{\Link}_{#1}^{#2}}

\def\Lktxx_#1^#2{\Link^{>0}_{#1,#2}}
\def\Starxx_#1^#2{\operatorname{Star}_{#1,#2}}
\def\Startxx_#1^#2{\operatorname{Star}^{\geq0}_{#1,#2}}

\def\sctnn_#1{\sc^{\geq0}_{#1}}

\def\sctp_#1^#2{\sc^{>0}_{#1,#2}}

\def\Seps_#1{S_{#1}}

\def\Lktpe_#1^#2{\Link^{>0}_{#1,#2}}
\def\Lktnne_#1{\Link^{\geq0}_{#1}}

\def\Lktp_#1^#2{\Link^{>0}_{#1,#2}}
\def\Lktnn_#1{\Link^{\geq0}_{#1}}

\def\sc{Z}

\def\sco_#1^#2{\accentset{\circ}{\sc}_{#1,#2}}

\def\sccl_#1^#2{\sc_{#1}^{#2}}

\def\Y{\mathcal{Y}}
\def\Yo_#1{\accentset{\circ}{\Y}_{#1}}
\def\Ycl_#1{\Y_{#1}}

\def\Ytp_#1{\Y_{#1}^{>0}}

\def\strg(#1){\normg{#1}}

\def\normg#1{\|#1\|}

\def\Jo{J_\bv^\circ}

\def\int{{\operatorname{init}}}

\def\Vi#1{v^{\parr{#1}}}

\def\Wi#1{w^{\parr{#1}}}
\def\vi#1{v_{\parr{#1}}}
\def\wi#1{w_{\parr{#1}}}

\def\PJ{P^J_-}

\def\pj{^{\parr{j}}}
\def\dom{\Delta^{\omega_i}_{\omega_i}}

\def\DOM^#1_#2{\Delta^{#1 \omega_i}_{#2 \omega_i}}
\def\DOMr^#1_#2{\Delta^{#1 \omega_r}_{#2 \omega_r}}
\def\DOMir^#1_#2{\Delta^{#1 \omega_{i_r}}_{#2 \omega_{i_r}}}
\def\om{\omega}

\def\line#1{\overline{#1}}
\def\lline#1{\overline{\overline{#1}}}

\def\sv{s^{\mathbf{v}}}

\newcommand*{\smallcap}{{\mathbin{\scalebox{0.5}{\ensuremath{\cap}}}}}%

\def\gc{g_\smallcap}
\def\BG{B_-\backslash G}
\def\BGtnn{(B_-\backslash G)_{\geq0}}

\def\RLtp_#1^#2{\cev R_{#1,#2}^{>0}}

\def\kbf{\mathbf{k}}

\def\Rsf_#1^#2{\Rich_{#1}^{#2}(K)}
\def\LRsf_#1^#2{\LRich_{#1}^{#2}(K)}
\def\GBsf{(G/B_-)(K)}
\def\BGsf{(\BG)(K)}

\def\pre{{\,\operatorname{pre}}}
\def\tpre_#1^#2{\vec\twistop^\pre_{#1,#2}}
\def\Ltpre_#1^#2{\cev\twistop^\pre_{#1,#2}}
\def\tpreL_#1^#2{\Ltpre_{#1}^{#2}}
\def\twistop{\tau}
\def\twist_#1^#2{\vec\twistop_{#1,#2}}
\def\twistL_#1^#2{\cev\twistop_{#1,#2}}

\def\bs{\backslash}
\def\Sn{S_n}
\def\chiv{\vec\chi_{v}}
\def\chivi{\cev\chi_{v}}

\def\chiw{\vec\chi^{\,w}}

\def\chiwi{\cev\chi^{\,w}}

\def\om{\omega}
\def\sv_#1{s^{\bv}_{#1}}
\def\gt{\gbf_{\bv,\bw}(\bt)}
\def\gtast{\gbf^\ast_{\bv,\bw}(\bt)}

\def\Lec{{\operatorname{Lec}}}

\def\fR{\vec f}
\def\fL{\cev f}

\def\capBWB(#1){(#1)^{\smallcap w}}
\def\capBWoB(#1){(#1)^{\smallcap w_0}}
\def\capBWBnopar(#1){#1^{\smallcap w}}

\def\th{\theta}
\def\yR{\vec y}
\def\yL{\cev y}

\def\Alec{\vec\Acal}

\def\PJm{P^J_-}
\def\Pio_#1^#2{\Pi^\circ_{#1,#2}}

\def\arr#1{%
{\scalebox{0.6}{%
\begin{tikzpicture}[scale=0.4,rotate=#1]%
\node[draw,circle,fill=black,scale=0.4](A) at(0,0){};%
\fill[black!25] (-0.5,-0.5)--(0,0)--(-0.5,0.5) to[bend right=20] (-0.5,-0.5);
\draw[line width=1.2pt](-0.5,-0.5)--(0,0)--(-0.5,0.5);%
\end{tikzpicture}%
}}%
}

\def\arrR{{\arr{0}}}
\def\arrL{{\arr{180}}}
\def\arrU{{\arr{90}}}
\def\arrD{{\arr{-90}}}

\def\RegL{C^\arrR}
\def\RegR{C^\arrL}
\def\RegD{C^\arrU}
\def\RegU{C^\arrD}

\def\UW{\wedge^\arrR}
\def\DW{\wedge^\arrL}

\def\DRv{I_{\shortrightarrow}^{\bv}}
\def\DRw{I_{\shortrightarrow}^{\bw}}
\def\ULv{I^{\shortleftarrow}_{\bv}}
\def\ULw{I^{\shortleftarrow}_{\bw}}

\def\chmnrR_#1{\vec\Delta_{#1}}
\def\chmnrL_#1{\cev\Delta_{#1}}

\def\BMX{B_-\backslash}

\def\Ypp{y'_+}
\def\Yp{y_+}
\def\Yo{y_0}
\def\Ym{y_-}
\def\Xggc{x}

\def\Xp{x_+}
\def\Xo{x_0}
\def\Xm{x_-}

\def\Htp{\H_{>0}}

\def\vv{v'}
\def\ww{w'}
\def\dvv{\dot v'}
\def\dww{\dot w'}
\def\H{H}

\def\Dpm_#1{\Delta^{\pm}_{#1}}
\def\Dmp_#1{\Delta^{\mp}_{#1}}

\def\ach{\alphacheck}

\def\pj{^{\parr{j}}}

\def\yRpr{\yR^{\,(r)}}
\def\yRpj{\yR^{\,(j)}}
\def\al{\alpha}
\def\QJ{Q^J}

\def\MS{\operatorname{MS}}
\def\TMSR_#1^#2{\vec \tau^{\,\MS}_{#1,#2}}
\def\TMSL_#1^#2{\cev \tau^{\,\MS}_{#1,#2}}

\def\IR{\vec I}
\def\IL{\cev I}

\def\Meas{\operatorname{Meas}}

\def\Vlax_#1{V(\la)_{#1}}

\def\QLec{Q_{\Lec}}

\def\Ing{\operatorname{Ing}}
\def\zing{z_{\Ing}}

\def\fLp{\fL'}
\def\QIng{Q_{\Ing}}

\makeatletter
\newcommand*\bigcdot{\mathpalette\bigcdot@{.4}}
\newcommand*\bigcdot@[2]{\mathbin{\vcenter{\hbox{\scalebox{#2}{$\m@th#1\bullet$}}}}}
\makeatother

\def\rev{\operatorname{rev}}
\def\revR{\vec\rev}
\def\revL{\cev\rev}
\def\laVd{{}^\la{\dot V}}
\def\Vd{\dot V}
\def\Bd{\dot {\bf B}(\la)}

\def\cevdotP{\text{\reflectbox{$\vec{\text{$\dot{\reflectbox{$P$\!}}$}}$}}\!}%

\def\PRds{\vec{\dot P}_{v,w}(\k)^\sigma}
\def\LPRds{\cevdotP_{v,w}(\k)^\sigma}
\def\LPRdsf{\cevdotP_{v,w}(K)^\sigma}

\def\Ztnn{\Z_{\geq0}}

\def\proofsection#1{%
\hypersetup{bookmarksdepth=-1}%
\subsubsection*{#1}%
\hypersetup{bookmarksdepth}%
}

\def\smat#1{
\begin{psmallmatrix}#1\end{psmallmatrix}
}

\makeatletter
\newcommand{\xMapsto}[2][]{\ext@arrow 0599{\Mapstofill@}{#1}{#2}}
\def\Mapstofill@{\arrowfill@{\Mapstochar\Relbar}\Relbar\Rightarrow}
\makeatother
\def\colspan{\operatorname{ColumnSpan}}

\makeatletter
\@namedef{subjclassname@2020}{%
  \textup{2020} Mathematics Subject Classification}
\makeatother
\def\yLT{\yL^{\,T}}
\begin{document}

\title{The twist for Richardson varieties}

\author{Pavel Galashin}
\address{Department of Mathematics, University of California, Los Angeles, 520 Portola Plaza,
Los Angeles, CA 90025, USA}
\email{\href{mailto:galashin@math.ucla.edu}{galashin@math.ucla.edu}}

\author{Thomas Lam}
\address{Department of Mathematics, University of Michigan, 2074 East Hall, 530 Church Street, Ann Arbor, MI 48109-1043, USA}
\email{\href{mailto:tfylam@umich.edu}{tfylam@umich.edu}}
\thanks{P.G.\ was supported by an Alfred P. Sloan Research Fellowship and by the National Science Foundation under Grants No.~DMS-1954121 and No.~DMS-2046915. T.L.\ was supported by Grants No.~DMS-1464693 and No.~DMS-1953852 from the National Science Foundation.}

\subjclass[2020]{
  Primary:
  14M15. %
  Secondary:
  13F60. %
}

\keywords{Twist map, Richardson variety, total positivity, cluster algebra, Grassmannian, positroid variety.}

\date{\today}

\begin{abstract}
We construct the twist automorphism of open Richardson varieties inside the flag variety of a complex semisimple algebraic group. We show that the twist map preserves totally positive parts, and prove a Chamber Ansatz formula for it.  Our twist map generalizes the twist maps previously constructed by Berenstein--Fomin--Zelevinsky, Marsh--Scott, and Muller--Speyer. We use it to explain the relationship between the two conjectural cluster structures for Richardson varieties studied by Leclerc and by Ingermanson.

\end{abstract}

\numberwithin{equation}{section}

\maketitle

\section{Introduction}

 Let $G$ be a complex, simply-connected, semisimple algebraic group, split over $\R$, and let $B,B_-\subset G$ be a pair of opposite Borel subgroups. Let $\H:=B\cap B_-$ be the maximal torus and $W:=N_G(\H)/\H$ be the associated Weyl group. We denote by $\ell(\cdot)$ the length function on $W$.   
 Let $N\subset B$ and $N_-\subset B_-$ be the unipotent radicals.
 We have Bruhat decompositions 
\begin{equation*}%
  G=\bigsqcup_{w\in W} B_-wB_-=\bigsqcup_{v\in W} BvB_-=\bigsqcup_{v\in W} B_-vB.
\end{equation*} 
We consider the \emph{(generalized) flag variety} $G/B_-$.
 Our main objects of study are the \emph{open Richardson varieties}
\begin{equation}\label{eq:Rich_dfn_R}
  \Rich_v^w:=(B_-wB_-\cap BvB_-)/B_- \subset G/B_-,%
\end{equation}
where the intersection is nonempty whenever $v\leq w$ in the Bruhat order on $W$. 

Our goal is to define the \emph{twist map} $\twist_v^w:\Rich_v^w\xrasim \Rich_v^w$ and study its properties. Previously, twist maps have been defined for the following varieties:
\begin{enumerate}[label=(\arabic*)]
\item\label{intro:case_unipotent} unipotent cells~\cite{Lus2,Lus3,BFZ,BZ};
\item\label{intro:case_double_Bruh} double Bruhat cells~\cite{FZ_double};
\item\label{intro:case_positroid} open positroid varieties in the Grassmannian~\cite{MaSc,MuSp}.
\end{enumerate}
Both unipotent cells and positroid varieties, as well as type $A$ double Bruhat cells, are special cases of open Richardson varieties. We show that $\twist_v^w$ specializes to the corresponding twist map in each of those cases. See also~\cite{GLS_generic,Williams_Q_syst,Weng_DT,Kimura} for related results on the twist map and its relation to the Donaldson--Thomas transformation.

In cases~\ref{intro:case_unipotent}--\ref{intro:case_positroid}, the twist map has a close relationship to \emph{total positivity}~\cite{gantmakher_krein_translation,Schoenberg,Lus2,LusIntro} and \emph{cluster algebras}~\cite{FZ}. The totally positive part $X_{>0}$ of each of the above varieties $X$ admits a \emph{positive parametrization}, that is, a homeomorphism $\gbf:(\R_{>0})^{\dim X}\xrasim X_{>0}$. This parametrization depends on some extra combinatorial data: a reduced word in case~\ref{intro:case_unipotent}, a double reduced word in case~\ref{intro:case_double_Bruh}, and a plabic graph~\cite{Pos} in case~\ref{intro:case_positroid}. Next, each variety $X$ in cases~\ref{intro:case_unipotent}--\ref{intro:case_positroid} admits a \emph{cluster algebra structure}~\cite{FZ3,Scott,GLS_quantum,SSBW,GL_cluster,GoYa}, which in particular includes a family of regular functions on $X$ called \emph{cluster variables}. The twist map $\twistop:X\xrasim X$ was shown to be an isomorphism that \emph{preserves total positivity}, that is, restricts to a homeomorphism $X_{>0}\xrasim X_{>0}$. Moreover, it allows one to invert the positive parametrization $\gbf:(\R_{>0})^{\dim X}\xrasim X_{>0}$ via a \emph{Chamber Ansatz} formula: for $\bt\in (\R_{>0})^{\dim X}$, the $\bt$-parameters are given by monomials in the cluster variables evaluated at $\twistop(\gbf(\bt))$.

We show that our twist map $\twist_v^w$ has analogous properties. (See \cref{sec:main} for the full description of our main results.) The totally positive part $\Rtp_v^w$ of $\Rich_v^w$ was defined by Lusztig~\cite{Lus2}; see also~\cite{Rietsch2}. Positive parametrizations  of $\Rtp_v^w$ were constructed in~\cite{MR}. They depend on the following combinatorial data: given a choice of a reduced expression $\bw$ for $w$, we find a \emph{positive distinguished subexpression} $\bv$ for $v$ inside $\bw$; see \cref{sec:MR_param}. This gives rise to a set $\Jo$ of size $|\Jo|=\ell(w)-\ell(v)=\dim\Rich_v^w$, and~\cite{MR} give a positive parametrization $\gbf_{\bv,\bw}:(\R_{>0})^{\Jo}\xrasim \Rtp_v^w$.
 Two conjectural cluster structures on $\Rich_v^w$ were studied in~\cite{Lec,Ing}.
\begin{statement}
  The map $\twist_v^w:\Rich_v^w\xrasim \Rich_v^w$ restricts to a homeomorphism $\Rtp_v^w\xrasim \Rtp_v^w$ on the totally positive part of $\Rich_v^w$. For $\bt\in (\R_{>0})^{\Jo}$, the $\bt$-parameters can be recovered as monomials in the regular functions considered in~\cite{MR,Lec,Ing}, evaluated at the twist of $\gt$.
\end{statement}

We note that one Chamber Ansatz formula, allowing one to recover the positive parameters of~\cite{MR}, has already been shown in~\cite{MR}. However, we found that this was only part of the complete picture. First, recall that unipotent cells considered in~\cite{BFZ,BZ} are special cases of $\Rich_v^w$ when $v=\id$. It turns out that in this case, the Chamber Ansatz formula of~\cite{BZ} is \emph{not} a specialization of the Chamber Ansatz formula of~\cite{MR}.
The explanation for this seemingly strange phenomenon comes from the work of Muller--Speyer~\cite{MuSp}: in the case of open positroid varieties, there are \emph{two} Chamber Ansatz formulas: one uses the \emph{right twist} map $\vec\twistop$~\cite{MuSp}, and the other one uses the \emph{left twist} $\cev\twistop$, which is the inverse map of $\vec\twistop$.
We show that the same situation takes place in the case of open Richardson varieties. In particular, we have two Chamber Ansatz formulas, fitting into the following commutative diagram, analogous to~\cite[Theorem~7.1]{MuSp}:
\begin{equation}\label{eq:twist_cd}
\begin{tikzcd}[column sep=85pt,row sep=35pt]
(\R_{>0})^{\Jo} \arrow[r,leftrightarrow,"\text{Left twist}","\text{Chamber Ansatz}"'] &(\R_{>0})^{\Jo}\arrow[r,leftrightarrow,"\text{Right twist}","\text{Chamber Ansatz}"'] \arrow[d,"\gbf_{\bv,\bw}"] &(\R_{>0})^{\Jo}\\
\Rtp_v^w \arrow[r,bend left=10,"\twist_v^w"] \arrow[u,"(\fL_j)_{j\in\Jo}"] &
\Rtp_v^w \arrow[l,bend left=10,"\twistL_v^w"] \arrow[r,bend left=10,"\twist_v^w"]& 
\Rtp_v^w. \arrow[l,bend left=10,"\twistL_v^w"]\arrow[u,"(\fR_j)_{j\in\Jo}"']
\end{tikzcd}  
\end{equation}
Here, all maps are homeomorphisms. The regular functions $\fR_j$ are products of cluster variables considered in~\cite{Lec}, while the regular functions $\fL_j$ are the \emph{chamber minors} considered in~\cite{MR}. (They are also closely related to the products of cluster variables considered in~\cite{Ing}.) The Chamber Ansatz of~\cite{BZ} corresponds to the right twist, while the Chamber Ansatz of~\cite{MR} corresponds to the left twist. 
 We have a complex-algebraic version of the above diagram:
\begin{equation}\label{eq:twist_cd_C}
\begin{tikzcd}[column sep=85pt,row sep=35pt]
\C^{\Jo} \arrow[r,dashed,leftrightarrow,"\text{Left twist}","\text{Chamber Ansatz}"'] &(\Cast)^{\Jo}\arrow[r,dashed,leftrightarrow,"\text{Right twist}","\text{Chamber Ansatz}"'] \arrow[d,"\gbf_{\bv,\bw}"] &\C^{\Jo}\\
\Rich_v^w \arrow[r,bend left=10,"\twist_v^w"] \arrow[u,"(\fL_j)_{j\in\Jo}"] &
\Rich_v^w \arrow[l,bend left=10,"\twistL_v^w"] \arrow[r,bend left=10,"\twist_v^w"]& 
\Rich_v^w. \arrow[l,bend left=10,"\twistL_v^w"]\arrow[u,"(\fR_j)_{j\in\Jo}"']
\end{tikzcd}  
\end{equation}
Here, the solid arrows denote regular maps and the horizontal dashed arrows denote rational monomial maps.

The original motivation for our work was to understand the relationship between the conjectural cluster structures on $\Rich_v^w$ considered in~\cite{Lec,Ing}. In \cref{sec:cluster}, we use the map $\twist_v^w$ to give a precise comparison.

We also consider the left-sided quotient $\BG$. Let
\begin{equation}\label{eq:Rich_dfn_L}
\LRich_v^w:=B_-\bs(B_-wB_-\cap B_-vB) \subset \BG.
\end{equation}
The intersection is again nonempty whenever $v\leq w$. One of the important ingredients of our construction is the \emph{chiral map}
\begin{equation}\label{eq:intro_chiral}
  \chiv: \LRich_v^w\xrasim \Rich_v^w,
\end{equation}
obtained by identifying both sides with $N\dv\cap \dv N\cap B_-wB_-$, where $N$ is the unipotent radical of $B$ and $\dv\in N_G(\H)$ is a lift of $v\in W$ to $G$. 
 We show that $\chiv$ preserves total positivity, i.e., restricts to a homeomorphism $\LRtp_v^w\xrasim \Rtp_v^w$. 
 The significance of the chiral map first occurred to us during the development of the results of~\cite{GKL3}. Later, this map played a key role in~\cite{GL_cluster}. In this paper, we use it in the construction of the twist map: by definition, $\twist_v^w$ is the composition of $\chiv$ with the \emph{right pre-twist map} $\tpre_v^w:\Rtp_v^w\xrasim \LRtp_v^w$.

\subsection*{Outline}
In \cref{sec:main}, we discuss our main results in detail. We review background material in \cref{sec:background}. 
 In \cref{sec:basic}, we show that the pre-twist and twist maps are well defined; cf. \cref{prop:pre_pre_twist}. 
 In \cref{sec:MR_chamber_ansatz}, we review the Chamber Ansatz of~\cite{MR} and use it to deduce the left twist Chamber Ansatz (\cref{thm:L_chamber_ans}). 
 We show that the chiral map\,---\,as well as some other related ``reversal'' maps\,---\,preserves total positivity (\cref{thm:chiral_tnn}) in \cref{sec:chiral}. 
 In \cref{sec:twist-TNN}, we prove the analogous statement for the pre-twist and the twist maps (\cref{thm:twist_tnn}).
 We prove the right twist Chamber Ansatz (\cref{thm:R_chamber_ans}) in \cref{sec:right-twist-chamber}.
 In \cref{sec:muller-speyer-twist}, we show that our twist map specializes to the twist of~\cite{MuSp} in the case of open positroid varieties (\cref{thm:MuSp_twist}). 
 We show that the chiral and the (pre-)twist maps are given by subtraction-free rational functions in \cref{sec:SF}.
 In \cref{sec:cluster}, we explain how our results connect the conjectural cluster structures on $\Rich_v^w$ studied in~\cite{Lec,Ing}. 
 Finally, in \cref{sec:fixed}, we discuss fixed points of the twist map.

\subsection*{Acknowledgments}
We thank Melissa Sherman-Bennett and David Speyer for interesting conversations related to this project, and George Lusztig for comments on an earlier version of this work.  We thank Steven Karp for collaboration on earlier related projects.

\section{Main results}\label{sec:main}
We explain our main results in detail. The reader is encouraged to refer to \cref{tab:formulas} for the notation and statements used throughout the paper, and to \cref{sec:A_example} for an example.

\subsection{Preliminaries} 
We fix a \emph{pinning} of $G$, which in particular includes a choice $\{\alpha_i\}_{i\in I}$ of simple roots for the root system $\Phi$ of $G$, and a homomorphism $\phi_i:\SL_2(\C)\to G$ for each $i\in I$. We set
\begin{equation*}%
  x_i(t):=\phi_i \begin{pmatrix}
1 & t \\ 0 & 1
  \end{pmatrix},\quad 
  y_i(t):=\phi_i \begin{pmatrix}
1 & 0 \\ t & 1
\end{pmatrix}, \quad\text{and}\quad
\ds_i:=\phi_i\begin{pmatrix}
0 & 1 \\ -1 & 0
\end{pmatrix}=x_i(1)y_i(-1)x_i(1).
\end{equation*}
We have an involutive anti-automorphism $g\mapsto g^T$ of $G$ defined by $x_i(t)^T=y_i(t)^T$, $y_i(t)^T=x_i(t)^T$, and $a^T=a$ for all $a\in \H$.

Lusztig~\cite{Lus2} defined the \emph{totally nonnegative part} $\Gtnn$ of $G$ as the submonoid generated by  $x_i(t),y_i(t)$ for $t>0$, as well as by the totally positive elements of $\H$; see \cref{sec:pinnings} for a precise definition. He defined the \emph{totally nonnegative flag variety} $\GBtnn$ to be the closure of the image of $\Gtnn$ under the projection $G\to G/B_-$. Similarly, define $\BGtnn\subset \BG$ to be the closure of the image of $\Gtnn$ under the projection $G\to\BG$. Letting $\Rtp_v^w:=\Rich_v^w\cap \GBtnn$ and $\LRtp_v^w:=\LRich_v^w\cap \BGtnn$, we have stratifications
\begin{equation*}%
  G/B_-=\bigsqcup_{v\leq w} \Rich_v^w,\quad \BG=\bigsqcup_{v\leq w} \LRich_v^w,\quad   \GBtnn=\bigsqcup_{v\leq w} \Rtp_v^w,\quad \BGtnn=\bigsqcup_{v\leq w} \LRtp_v^w.
\end{equation*}

For each element $w\in W$, a \emph{reduced word} for $w$ is an expression of the form $w=s_{i_1}s_{i_2}\cdots s_{i_m}$, where $m$ is minimal possible. We refer to such $m$ as the \emph{length} of $w$ and denote $\ell(w):=m$. Given such a reduced word, we set $\dw:=\ds_{i_1}\ds_{i_2}\cdots \ds_{i_m}$. The resulting element $\dw\in G$ does not depend on the choice of the reduced word. We omit the dot when the choice of the representative of $w$ inside $N_G(\H)$ does not matter; for example, we write $B_-wB_-$ instead of $B_-\dw B_-$.

\subsection{Type $A_{n-1}$}
Let $G=\SL_n(\C)$, in which case $B$, $B_-$, $N$, $N_-$, and $\H$ are the subgroups consisting of upper triangular, lower triangular, upper unitriangular, lower unitriangular, and diagonal matrices, respectively. The Weyl group $W$ is the symmetric group $\Sn$. For $j\in[n]:=\{1,2,\dots,n\}$ and a permutation $w=s_{i_1}\cdots s_{i_m}$, we set $w(j):=s_{i_1}(\cdots(s_{i_m}(j))\cdots )$.

The submonoid $\Gtnn$ consists of matrices all of whose minors (of all sizes) are nonnegative. For two sets $I,J\subset[n]$ of the same size $k$, let $\Delta^I_J$ denote the $k\times k$ minor with row set $I$ and column set $J$. We say that $\Delta^I_J$ is a \emph{right-aligned flag minor} if $J=\{n-k+1,\dots,n-1,n\}$ consists of the rightmost $k$ columns. Similarly, $\Delta^I_J$ is a \emph{top-aligned flag minor} if $I=[k]$ consists of the top $k$ rows.

For a matrix $g\in G$, the flag $gB_-$ belongs to $\GBtnn$ if and only if the right-aligned flag minors of $g$ are all nonnegative. (More precisely, if and only if for each $k\in [n]$, the ratio of any two $k\times k$ nonzero right-aligned flag minors of $g$ is nonnegative.) Similarly, $B_-g$ belongs to $\BGtnn$ if and only if the top-aligned flag minors of $g$ are all nonnegative. For a permutation $v\in W$, the matrix $\dv\in G$ is a signed permutation matrix for $v$, where the signs are chosen so that the right-aligned (equivalently, the top-aligned) flag minors of $\dv$ are nonnegative. For example, if $v=s_1s_2=[2,3,1,4]$ in one-line notation then 
\begin{equation}\label{eq:example_v=3124}
  \dv=\ds_1\ds_2=\smat{
0 & 0 & 1 & 0\\
-1 & 0 & 0 & 0\\
0 & -1 & 0 & 0\\
0 & 0 & 0 & 1
}.
\end{equation}

\subsection{The chiral map}
Let $v\leq w$ in $W$. Our first main result is a construction of a natural map relating $\LRich_v^w\subset\BG$ and $\Rich_v^w\subset G/B_-$.
\begin{notation}
For a subset $F\subset G$, we set $\capBWBnopar(F):=F\cap B_-wB_-$.
\end{notation}

It is known~\cite{BGY,Lec} that the maps $B_-g\mapsfrom g\mapsto gB_-$ restrict to isomorphisms
\begin{equation}\label{eq:NvN_to_Rich}
 \LRich_v^w\Lxrasim \capBWB(N\dv \cap \dv N) \xrasim \Rich_v^w.
\end{equation}
\begin{definition}
The \emph{chiral map} $\chiv: \LRich_v^w\xrasim \Rich_v^w$ is obtained by identifying both spaces with $\capBWB(N\dv \cap \dv N)$ via~\eqref{eq:NvN_to_Rich}. %
 The inverse of $\chiv$ is denoted $\chivi$.
\end{definition}
\noindent Taking the union over all $w\in W$ satisfying $w\geq v$, we get an isomorphism 
\begin{equation*}%
  \chiv:B_-\bs(B_-vB) \xrasim (BvB_-)/B_-
\end{equation*}
 on the \emph{opposite Schubert cells} obtained by identifying both spaces with $N\dv \cap \dv N$. 
  We state our first result: the chiral map preserves total positivity.
\begin{theorem}\label{thm:chiral_tnn}
Let $v\in W$. Then for any $g\in N\dv \cap \dv N$, we have
\begin{equation}\label{eq:chiral_tnn}
  gB_-\in \GBtnn \quad \Longleftrightarrow \quad B_-g\in\BGtnn.
\end{equation}
In other words, for all $v\leq w \in W$, the map $\chiv: \LRich_v^w\xrasim \Rich_v^w$ restricts to a homeomorphism
\begin{equation*}%
  \chiv: \LRtp_v^w\xrasim \Rtp_v^w.
\end{equation*}
\end{theorem}
Consider the case $G=\SL_n(\C)$. Then the set $N\dv\cap \dv N$ consists of matrices obtained from the signed permutation matrix $\dv$ by allowing nonzero entries in the cells which are both to the right and above a $\pm1$ in $\dv$. E.g., for $v=[2,3,1,4]$ and $\dv$ as in~\eqref{eq:example_v=3124}, the set $N\dv\cap \dv N$ consists of matrices
\begin{equation}\label{eq:example_v=3124_matrix}
\smat{
0 & 0 & 1 & a\\
-1 & -d & 0 & b\\
0 & -1 & 0 & c\\
0 & 0 & 0 & 1
} \quad\text{for $a,b,c,d\in\C$.}
\end{equation}
 Then \cref{thm:chiral_tnn} claims that a matrix $g\in N\dv\cap \dv N$ has all right-aligned flag minors nonnegative if and only if it has all top-aligned flag minors nonnegative. We encourage the reader to check that either one of these two conditions is satisfied for the matrix in~\eqref{eq:example_v=3124_matrix} whenever $a,b,c,d\in\R$ satisfy
\begin{equation}\label{eq:abcd>=0}
  a,b,c,d,cd-b\geq 0.
\end{equation}

\subsection{The twist map} Let $\Gopm:=BB_-$ and $\Gomp:=B_-B$. It is well known that the multiplication map gives isomorphisms
\begin{equation}\label{eq:Gopm_Gomp}
  N\times \H\times N_-\xrasim \Gopm \quad\text{and}\quad N_-\times \H\times N\xrasim \Gomp.
\end{equation}
We denote the inverses of these maps by $g\mapsto ([g]_L^+,[g]_0^\pm,[g]_R^-)$ and $g\mapsto ([g]_L^-,[g]_0^\mp,[g]_R^+)$, respectively. Here $[g]_L^+,[g]_R^+\in N$, $[g]_R^-,[g]_L^-\in N_-$, and $[g]_0^\pm,[g]_0^\mp\in \H$. In type $A$, the factorizations
\begin{equation}\label{eq:UDL_LDU}
  g=[g]_L^+\cdot [g]_0^\pm\cdot [g]_R^- \quad\text{and}\quad g=[g]_L^-\cdot [g]_0^\mp\cdot [g]_R^+
\end{equation}
are known as the \emph{UDL and LDU decompositions}, respectively. We will be mostly interested in the elements $[g]_L^+$ and $[g]_R^+$.

The twist map essentially boils down to the following isomorphism.
\begin{proposition}\label{prop:pre_pre_twist} 
Let $v\leq w$ in $W$. We have an isomorphism 
\begin{equation*}%
  \capBWB(N\dv)\xrasim \capBWB(\dv N),\quad g\mapsto \dv[g^T\dw]_L^+.
\end{equation*}
The inverse is given by $h\mapsto [\dw h^T]_R^+\dv$.
\end{proposition}
It turns out that the map of \cref{prop:pre_pre_twist} descends to a map $\Rich_v^w\xrasim \LRich_v^w$. We refer to the resulting map as the \emph{right pre-twist}.
\begin{definition}\label{dfn:pre_twist}
Define the \emph{right pre-twist} $\tpre_v^w$ and the \emph{left pre-twist} $\Ltpre_v^w$ by
\begin{align*}%
  \tpre_v^w:\Rich_v^w&\xrasim \LRich_v^w,\quad gB_- \mapsto B_-v[g^T\dw]_L^+ \quad\text{for $g\in \capBWB(N\dv)$};\\
  \Ltpre_v^w:\LRich_v^w&\xrasim \Rich_v^w,\quad B_-h\mapsto [\dw h^T]_R^+v B_- \quad\text{for $h\in \capBWB(\dv N)$}.
\end{align*}
The \emph{right twist} $\twist_v^w$ and the \emph{left twist} $\twistL_v^w$ are defined by composing the pre-twists with the chiral map:
\begin{equation*}%
  \twist_v^w:=\chiv\circ \tpre_v^w:\Rich_v^w\xrasim \Rich_v^w,\qquad \twistL_v^w:=\Ltpre_v^w \circ \chivi:\Rich_v^w\xrasim \Rich_v^w.
\end{equation*}
\end{definition}
It follows from \cref{prop:pre_pre_twist} that the maps $\twist_v^w$ and $\twistL_v^w$ are mutually inverse automorphisms of $\Rich_v^w$. It turns out that these maps also preserve total positivity.

\begin{theorem}\label{thm:twist_tnn}
  For all $v\leq w$ in $W$, the twist map restricts to a homeomorphism
\begin{equation*}%
  \twist_v^w:\Rtp_v^w\xrasim \Rtp_v^w.
\end{equation*}
\end{theorem}
\noindent See \cref{sec:A_example} for an example.

In fact, the twist map and the chiral map have the property of being given by \emph{subtraction-free rational expressions}; see \cref{sec:SF}.

\subsection{Muller--Speyer twist}\label{sec:MuSp_main_res}
Let $G=\SL_n(\C)$, and thus $I=[n-1]$. Choose $k\in I$ and set $J:=I\setminus\{k\}$. Let $\PJm\subset G$ be the parabolic subgroup consisting of block lower triangular matrices of the form $\begin{pmatrix}
A & 0\\
C & D
\end{pmatrix}$, where $A\in\GL_k(\C)$ and $D\in\GL_{n-k}(\C)$. The quotient $G/\PJm$ is isomorphic to the Grassmannian $\Gr(n-k,n)$ of $(n-k)$-planes inside $\C^n$ via the map sending $g\PJm\in G/\PJm$ to the column span of the rightmost $n-k$ columns of $g$. We say that $w\in \Sn$ is \emph{$k$-Grassmannian} if $w(1)<\cdots <w(k)$ and $w(k+1)<\cdots<w(n)$, and we let $W^J$ denote the set of $k$-Grassmannian permutations. Let $Q^J:=\{(v,w)\in W\times W^J\mid v\leq w\}$. Then the elements of $Q^J$ label \emph{open positroid varieties}~\cite{KLS} inside $\Gr(n-k,n)$. Specifically, for each $(v,w)\in Q^J$, the natural projection map $\pi_J:G/B_-\to G/\PJm$ restricts to an isomorphism $\Rich_v^w\xrasim \Pio_v^w$, where we set $\Pio_v^w:=\pi_J(\Rich_v^w)$. We have a stratification
\begin{equation*}%
  \Gr(n-k,n)=\bigsqcup_{(v,w)\in Q^J} \Pio_v^w.
\end{equation*}

For each $(v,w)\in Q^J$, Muller and Speyer~\cite{MuSp} defined the \emph{right} and \emph{left twist} automorphisms of $\Pio_v^w$. Even though their definition is quite different from ours, we show that their twist maps are special cases of our twist maps. 
\begin{theorem}\label{thm:MuSp_twist}
  For all $(v,w)\in Q^J$, under the identification $\pi_J:\Rich_v^w\xrasim \Pio_v^w$, the maps $\twist_v^w$ and $\twistL_v^w$ coincide respectively with the right and left twist maps defined in~\cite{MuSp}.
\end{theorem}
\noindent See \cref{ex:MuSp}. In particular, up to the action of $\H$, our twist maps generalize the twist maps of Marsh--Scott~\cite{MaSc}. In view of~\cite[Section~A.4]{MuSp}, we also obtain the twist map of~\cite{FZ_double} for type $A$ double Bruhat cells as a special case.

\subsection{MR-parametrizations}\label{sec:MR_param}

Fix $v\leq w$ in $W$. Let $w=s_{i_1}s_{i_2}\cdots s_{i_m}$ be a reduced word for $w\in W$.  In this case, $\bw:=(i_1,\dots,i_m)$ is called a \emph{reduced expression} for $w$.  
 By~\cite[Lemma~3.5]{MR}, every reduced expression $\bw$ for $w$ contains a unique ``rightmost'' reduced subexpression $\bv$ for $v$, called the \emph{positive distinguished subexpression}. We denote the set of such pairs $(\bv,\bw)$ by $\Red(v,w)$, and for $(\bv,\bw)\in\Red(v,w)$, we let $\Jo\subset [m]$ denote the set of indices \emph{not} used in $\bv$.

The space $\Rich_v^w$ contains an open dense subset, called the \emph{open Deodhar stratum}. It admits the following \emph{Marsh--Rietsch (MR) parametrization}, introduced in \cite{MR}. For $\bt=(t_r)_{r\in\Jo} \in (\C^\ast)^{|\Jo|}$, define an element 
\begin{equation}\label{eq:MR}
\gt  =  g_1\cdots g_m \in \capBWB(N\dv),
\qquad \text{where} \qquad
g_r = \begin{cases} \ds_{i_r} & \mbox{if $r \notin \Jo$,} \\
x_{i_r}(t_r) &\mbox{if $r \in \Jo$.} \end{cases}
\end{equation}
The map $(\C^\ast)^{\Jo} \to \Rich_v^w$ given by $\t \mapsto \gt B_-$ is an isomorphism onto its image, the open Deodhar stratum in $\Rich_v^w$. The restriction of this map to $(\R_{>0})^{\Jo}$ gives a homeomorphism $(\R_{>0})^{\Jo}\xrasim \Rtp_v^w$.

\subsection{Chamber Ansatz}

For $r\in[m]$, we set 
\[\wi r:=s_{i_1}\cdots s_{i_r},\qquad\vi r:=\sv_{i_1}\cdots\sv_{i_r},\quad\text{where}\quad\sv_{i_r}:=
\begin{cases}
  s_{i_r}, &\text{if $r\notin\Jo$,}\\
  1,&\text{if $r\in \Jo$;}
\end{cases}\]
\[
\Wi r:=w^{-1}\cdot \wi r=s_{i_m}\cdots s_{i_{r+1}},\qquad \text{and}\qquad \Vi r:=v^{-1}\cdot \vi r=\sv_{i_m}\cdots\sv_{i_{r+1}}.
\]

Let $\{\om_i\}_{i\in I}$ be the fundamental weights corresponding to the simple roots $\{\alpha_i\}_{i\in I}$. For each $i\in I$ and $v,w\in W$, we have a regular function $\DOM^v_w: G\to \C$, called a \emph{generalized minor}~\cite{FZ_double}. The function $\DOM^v_w$ depends only on the weights $v\om_i,w\om_i$. In the case $G=\SL_n(\C)$, $\DOM^v_w$ is the usual matrix minor with row set $\{v(1),\dots,v(i)\}$ and column set $\{w(1),\dots,w(i)\}$.

For $r\in [m]$, set
\begin{equation}\label{eq:fR_dfn}
  \fR_r:=\DOMir^{v\pu{r-1}}_{w\pu{r-1}}: N\to \C.
\end{equation}
Leclerc~\cite{Lec} shows that the irreducible factors of the generalized minors $\fR_1,\fR_2,\dots,\fR_m$ form a seed of a cluster subalgebra $\Alec\subset\C[\Rich_v^w]$ of the coordinate ring of $\Rich_v^w$.

Let $(a_{i,j})_{i,j\in I}$ denote the Cartan matrix of the root system $\Phi$ of $G$, given by $a_{i,j}:=\<\alpha_i,\alpha_j^\vee\>$; see \cref{sec:pinnings} for further details. We prove the following result.

\begin{theorem}[Right twist Chamber Ansatz]\label{thm:R_chamber_ans}
Let $v\leq w$ in $W$ and $(\bv,\bw)\in\Red(v,w)$. Let $g:=\gt$ be MR-parametrized, and let $\yR:=[g^T\dw]_L^+$ (cf. \cref{prop:pre_pre_twist}). Then for all $r\in\Jo$, we have
\begin{equation}\label{eq:R_chamber_ans}
  t_r=\frac{ \prod_{j\neq i_r} \Delta^{\Vi r\om_j}_{\Wi r\om_j}(\yR)^{-a_{j,i_r}}  }
{\Delta^{\Vi r\om_{i_r}}_{\Wi r\om_{i_r}}(\yR) \Delta^{\Vi{r-1}\om_{i_r}}_{\Wi{r-1}\om_{i_r}}(\yR)   }.
\end{equation}
Conversely, for each $j\in[m]$, we have 
\begin{equation}\label{eq:R_chamber_ans_inv}
  \fR_j(\yR)=\prod_{r\in\Jo:\, r\geq j } t_r^{-\<\om_{i_j},s_{i_j}s_{i_{j+1}}\cdots s_{i_{r-1}} \ach_{i_r}\>}.
\end{equation}
\end{theorem}
\begin{remark}\label{rmk:intro:all_minors_f_r}
Each non-trivial (i.e., not identically equal to $1$) generalized minor of $\yR$ appearing in the right-hand side of~\eqref{eq:R_chamber_ans} is of the form $\fR_j(\yR)$ for some $j\in [m]$.  As we explain in Section~\ref{ssec:Lec}, $\fR_j(\yR)$ is a function of the right twist $\twist_v^w(gB_-)$.  In particular,~\eqref{eq:R_chamber_ans} gives a way to recover the MR-parameters $(t_r)_{r\in \Jo}$ from $\twist_v^w(gB_-)$.  See also~\cite[Equation~(4-1)]{BaoHe}.
\end{remark}
\begin{remark}
Let $v=\id$. Then $\Rich_v^w\cong \capBWBnopar(N)$ is a unipotent cell, considered in~\cite{BFZ,BZ}. It is clear that the twist map $\eta_w: \capBWBnopar(N)\xrasim \capBWBnopar(N)$ defined in~\cite[Theorem~1.2]{BZ} coincides with the map $h\mapsto [\dw h^T]_R^+$ of \cref{prop:pre_pre_twist}. The Chamber Ansatz in~\cite[Theorem~1.4]{BZ} describes the $\bt$-parameters in terms of the functions $\fR_j$ evaluated at the element $\eta_w^{-1}(g)=\yR$. Thus~\cite[Theorem~1.4]{BZ} is a special case of \cref{thm:R_chamber_ans}.
\end{remark}
\begin{remark}
Even though~\eqref{eq:R_chamber_ans} looks almost identical to~\cite[Theorem~7.1(1)]{MR}, the two Chamber Ansatz formulas are genuinely different, and~\eqref{eq:R_chamber_ans} cannot be deduced from the results of~\cite{MR} in a simple way. The crucial difference between~\eqref{eq:R_chamber_ans} and~\cite[Theorem~7.1(1)]{MR} is that our formula involves $\Vi r,\Wi r$ while~\cite[Theorem~7.1(1)]{MR} involves $\vi r,\wi r$ instead. 
 See \cref{fig:ch_ans} for a comparison.
\end{remark}

In order to complete the diagrams in~\eqref{eq:twist_cd}--\eqref{eq:twist_cd_C}, we need to state the \emph{left twist Chamber Ansatz}, which is much closer to the results of~\cite{MR}. Note that we work inside $G/B_-$ while~\cite{MR} work inside $G/B$.  To relate our results with those in~\cite{MR}, we apply the involution $x\mapsto x^\th$ of~\cite[Eq.~(1.11)]{FZ}; see \cref{rmk:inv_TNN}. 

Let $w_0\in W$ be the longest element of $W$. For each $i\in I$, let $i^\ast\in I$ be defined by $\om_{i^\ast}=-w_0\om_i$. 
Set 
\begin{equation}\label{eq:fL_dfn}
  \fL_r:=\Delta^{\wi r w_0\om_{i_r^\ast}}_{\vi r w_0\om_{i_r^\ast}}: N\to \C.
\end{equation}

\begin{theorem}[Left twist Chamber Ansatz~\cite{MR}]\label{thm:L_chamber_ans}
Let $v\leq w$ in $W$ and $(\bv,\bw)\in\Red(v,w)$. Let $g:=\gt\in\capBWB(N\dv)$ be MR-parametrized, and let $\gc\in \capBWB(N\dv\cap\dv N)$ be the unique element satisfying $gB_-=\gc B_-$. Let $\yL:=[\dw \gc^T]_R^+$ (cf. \cref{prop:pre_pre_twist}). Then for all $r\in\Jo$, we have
\begin{equation}\label{eq:L_chamber_ans}
  t_r=\frac{ \prod_{j\neq i_r^\ast} \Delta^{\wi r w_0\om_j}_{\vi rw_0\om_j}(\yL)^{-a_{j,i_r^\ast}}  }
{\Delta^{\wi r w_0\om_{i_r^\ast}}_{\vi r w_0\om_{i_r^\ast}}(\yL) \Delta^{\wi{r-1}w_0\om_{i_r^\ast}}_{\vi{r-1}w_0\om_{i_r^\ast}}(\yL)   }.
\end{equation}
Conversely, for each $j\in[m]$, we have 
\begin{equation}\label{eq:L_chamber_ans_inv}
  \fL_j(\yL)=\prod_{r\in\Jo:\, r\leq j } t_r^{-\<\om_{i_j},s_{i_j}s_{i_{j-1}}\cdots s_{i_{r+1}} \ach_{i_r}\>}.
\end{equation}
\end{theorem}
\noindent Similarly to \cref{rmk:intro:all_minors_f_r}, all non-trivial generalized minors in the right-hand side of~\eqref{eq:L_chamber_ans} are of the form $\fL_j(\yL)$ for some $j\in[m]$. 
\begin{remark}
After switching from $G/B_-$ to $G/B$, the left twist Chamber Ansatz~\eqref{eq:L_chamber_ans} coincides with~\cite[Theorem~7.1(1)]{MR}, except that the element $\yL$ is replaced with the transpose $z^T$ of an arbitrary element $z\in N_-$ such that $zwB_-=gB_-$; see \cref{thm:MR_chamber_ans}. We show in \cref{sec:MR_chamber_ansatz} that the element $z:=\yLT$ satisfies these assumptions, and thus deduce~\eqref{eq:L_chamber_ans} from the results of~\cite{MR}.
\end{remark}

\subsection{Type $A$ Chamber Ansatz}
When $G=\SL_n(\C)$, our results admit a pictorial interpretation in terms of wiring diagrams. Fix $v\leq w$ in $W$ and $(\bv,\bw)\in\Red(v,w)$. We construct the \emph{wiring diagram} of $\bw=(i_1,i_2,\dots,i_m)$ by drawing $n$ horizontal strands and placing, for each $r\in[m]$, a crossing with horizontal coordinate $r$ between the strands with vertical coordinates $i_r$ and $i_r+1$. We refer to $i_r$ as the \emph{height} of this crossing. For each $r\in\Jo$, we place a dot labeled $t_r$ on the corresponding crossing. The connected components of the complement of the wiring diagram are called \emph{chambers}. The \emph{height} of a chamber is the number of strands passing below it. For $r\in [m]$, we denote by $\RegR_r,\RegU_r,\RegL_r,\RegD_r$ the chambers located immediately to the right, above, to the left, and below the $r$-th crossing, respectively:%
\begin{center}
  \includegraphics[width=0.18\textwidth]{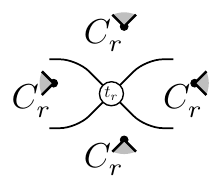}
\end{center}

We refer to the strands in the wiring diagram of $\bw$ as \emph{$\bw$-strands}. A \emph{$\bv$-strand} is a path from left to right that follows the crossings not in $\Jo$ and ignores the crossings in $\Jo$. The right and left endpoints of the strands are labeled $1,2,\dots,n$ from the bottom to the top. For a chamber $C$, we define subsets $\DRv(C),\DRw(C),\ULv(C),\ULw(C)\subset[n]$ by
\begin{align*}
\DRv(C)&:=\{\text{right endpoints of $\bv$-strands passing below $C$}\}; \\
\DRw(C)&:=\{\text{right endpoints of $\bw$-strands passing below $C$}\};\\
\ULv(C)&:=\{\text{left endpoints of $\bv$-strands passing above $C$}\};\\
\ULw(C)&:=\{\text{left endpoints of $\bw$-strands passing above $C$}\}.
\end{align*}
Finally, we introduce \emph{chamber minors}
\begin{equation*}%
  \chmnrR_C:=\Delta^{\DRv(C)}_{\DRw(C)} \quad\text{and}\quad   \chmnrL_C:=\Delta^{\ULw(C)}_{\ULv(C)}.
\end{equation*}
Comparing to~\eqref{eq:fR_dfn} and~\eqref{eq:fL_dfn}, we see that for all $r\in [m]$, we have
\begin{equation}\label{eq:fR_fL_chmnr}
  \fR_r=\chmnrR_{\RegL_r} \quad\text{and}\quad \fL_r=\chmnrL_{\RegR_r}.
\end{equation}

\begin{remark}
We emphasize that there is no symmetry between $\chmnrR_C$ and $\chmnrL_C$ since both of them are defined in terms of the \emph{rightmost} subexpression for $v$ inside $\bw$.
\end{remark}

Let $r\in[m]$ and consider the two $\bw$-strands $S_1,S_2$ participating in the $r$-th crossing. Let $\UW_r$ (resp., $\DW_r$) denote the union of chambers located vertically between $S_1$ and $S_2$ to the left (resp., to the right) of the $r$-th crossing. Following~\cite{MuSp}, we refer to $\UW_r$ and $\DW_r$ as the \emph{upstream wedge} and the \emph{downstream wedge}, respectively. The results of the previous subsection can be specialized as follows.

\begin{theorem}[Type $A$ Chamber Ansatz]\label{thm:A_ch_ans}
Let $G=\SL_n(\C)$, $v\leq w$ in $W$, and $(\bv,\bw)\in\Red(v,w)$. Let $g:=\gt$ and $\yR:=[g^T\dw]_L^+$ be as in \cref{thm:R_chamber_ans}. Let $\gc\in N\dv\cap\dv N$ with $gB_-=\gc B_-$ and $\yL:=[\dw \gc^T]_R^+$ be as in \cref{thm:L_chamber_ans}. Then for all $r\in\Jo$,
\begin{equation}\label{eq:A_ch_ans}
  t_r=\frac{\chmnrR_{\RegU_r}(\yR)\chmnrR_{\RegD_r}(\yR)}{\chmnrR_{\RegL_r}(\yR)\chmnrR_{\RegR_r}(\yR)} = \frac{\chmnrL_{\RegU_r}(\yL)\chmnrL_{\RegD_r}(\yL)}{\chmnrL_{\RegL_r}(\yL)\chmnrL_{\RegR_r}(\yL)}.
\end{equation}
Conversely, for each chamber $C$ of the wiring diagram of $\bw$, we have 
\begin{equation*}%
  \chmnrR_C(\yR)=\prod_{r\in\Jo:\,C\in\UW_r} t_r^{-1} \quad\text{and}\quad \chmnrL_C(\yL)=\prod_{r\in\Jo:\,C\in\DW_r} t_r^{-1}.
\end{equation*}
\end{theorem}

\subsection{Type $A$ example}\label{sec:A_example}
Let $G=\SL_n(\C)$ for $n=4$, and let $v=s_1s_2$ as in~\eqref{eq:example_v=3124}. Let $w=w_0$, and let $\bw=(2,1,2,3,2,1)$ be a reduced expression for $w$. We have
\begin{equation*}%
  g=\gt=x_2(t_1)\ds_1 x_2(t_3)x_3(t_4)\ds_2 x_1(t_6)=
\smat{
0 & -t_{3} & 1 & t_{3} t_{4} \\
-1 & -t_{1} - t_{6} & 0 & t_{1} t_{4} \\
0 & -1 & 0 & t_{4} \\
0 & 0 & 0 & 1
}, \quad 
\gc=\smat{
0 & 0 & 1 & t_{3} t_{4} \\
-1 & -t_{1} - t_{6} & 0 & t_{1} t_{4} \\
0 & -1 & 0 & t_{4} \\
0 & 0 & 0 & 1
},
\end{equation*}
where we assume $t_r>0$ for all $r\in\Jo=\{1,3,4,6\}$. In terms of the $(a,b,c,d)$-parameters in~\eqref{eq:example_v=3124_matrix}, we have $a=t_3t_4$, $b=t_1t_4$, $c=t_4$, and $d=t_1+t_6$. 
Computing the UDL decomposition of $g^T\dw$, we get
\begin{equation}\label{eq:ex_YpYoYm}
  g^T\dw=\smat{
0 & 0 & 1 & 0 \\
0 & -1 & t_{1} + t_{6} & -t_{3} \\
0 & 0 & 0 & 1 \\
-1 & t_{4} & -t_{1} t_{4} & t_{3} t_{4}
}=\smat{
1 & \frac{1}{t_{6}} & \frac{t_{3}}{t_{1}} & 0 \\
0 & 1 & \frac{t_{3} t_{6}}{t_{1}} & -\frac{1}{t_{4}} \\
0 & 0 & 1 & \frac{1}{t_{3} t_{4}} \\
0 & 0 & 0 & 1
}\cdot \smat{
\frac{1}{t_{4} t_{6}} & 0 & 0 & 0 \\
0 & \frac{t_{6}}{t_{1}} & 0 & 0 \\
0 & 0 & \frac{t_{1}}{t_{3}} & 0 \\
0 & 0 & 0 & t_{3} t_{4}
}\cdot \smat{
1 & 0 & 0 & 0 \\
-\frac{t_{1} + t_{6}}{t_{4} t_{6}} & 1 & 0 & 0 \\
\frac{1}{t_{1} t_{4}} & -\frac{1}{t_{1}} & 1 & 0 \\
-\frac{1}{t_{3} t_{4}} & \frac{1}{t_{3}} & -\frac{t_{1}}{t_{3}} & 1}.
\end{equation}
Thus we have
\begin{equation*}%
  \yR=[g^T\dw]_L^+= \smat{
1 & \frac{1}{t_{6}} & \frac{t_{3}}{t_{1}} & 0 \\
0 & 1 & \frac{t_{3} t_{6}}{t_{1}} & -\frac{1}{t_{4}} \\
0 & 0 & 1 & \frac{1}{t_{3} t_{4}} \\
0 & 0 & 0 & 1
} \quad\text{and}\quad \dv\yR=\smat{
0 & 0 & 1 & \frac{1}{t_{3} t_{4}} \\
-1 & -\frac{1}{t_{6}} & -\frac{t_{3}}{t_{1}} & 0 \\
0 & -1 & -\frac{t_{3} t_{6}}{t_{1}} & \frac{1}{t_{4}} \\
0 & 0 & 0 & 1
}.
\end{equation*}
The top-aligned flag minors of $\dv\yR$ give the  pre-twist $\tpre_v^w(gB_-)=B_-v\yR$. To apply the chiral map to $B_-v\yR$, we need to apply downward row operations to $\dv\yR\in\capBWB(\dv N)$ to get an element $y_1\in \capBWB(N\dv\cap \dv N)$. This element is given by
\begin{equation*}%
  y_1=\smat{
0 & 0 & 1 & \frac{1}{t_{3} t_{4}} \\
-1 & -\frac{1}{t_{6}} & 0 & \frac{1}{t_{1} t_{4}} \\
0 & -1 & 0 & \frac{t_{1} + t_{6}}{t_{1} t_{4}} \\
0 & 0 & 0 & 1
},\quad\text{thus} \quad \twist_v^w(gB_-)=\chiv(B_-v\yR)=\chiv(B_-y_1)=y_1B_-.
\end{equation*}
Notice that each right-aligned flag minor of $y_1$ is positive (and in fact, subtraction-free as a function of $t_1,t_3,t_4,t_6$). 
We may also compute the left twist: we have $\chivi(gB_-)=B_-\gc$, 
\begin{equation*}%
 \yL=[\dw \gc^T]_R^+=\smat{
1 & \frac{t_{1}}{t_{3}} & \frac{1}{t_{3}} & \frac{1}{t_{3} t_{4}} \\
0 & 1 & \frac{1}{t_{1}} & \frac{1}{t_{1} t_{4}} \\
0 & 0 & 1 & \frac{t_{1} + t_{6}}{t_{4} t_{6}} \\
0 & 0 & 0 & 1
}, \quad\text{and}\quad \twistL_v^w(gB_-)=\yL vB_-=\smat{
-\frac{t_{1}}{t_{3}} & -\frac{1}{t_{3}} & 1 & \frac{1}{t_{3} t_{4}} \\
-1 & -\frac{1}{t_{1}} & 0 & \frac{1}{t_{1} t_{4}} \\
0 & -1 & 0 & \frac{t_{1} + t_{6}}{t_{4} t_{6}} \\
0 & 0 & 0 & 1
}\cdot B_-.
\end{equation*}
In the $(a,b,c,d)$-coordinates in~\eqref{eq:example_v=3124_matrix}, the right and the left twist maps are given by
\begin{equation*}%
  (a,b,c,d) \xmapsto{\twist_v^w} \left(\frac1a,\frac1b, \frac db, \frac c{cd-b}\right),\quad (a,b,c,d) \xmapsto{\twistL_v^w}\left(\frac1a,\frac1b, \frac d{cd-b}, \frac cb\right).
\end{equation*}
These two maps are mutually inverse and preserve total positivity; cf.~\eqref{eq:abcd>=0}. Finally, we check and compare the two Chamber Ansatz formulas (\cref{thm:A_ch_ans}) in \cref{fig:ch_ans}.

\begin{figure}
\begin{tabular}{c}%
\includegraphics[width=0.98\textwidth]{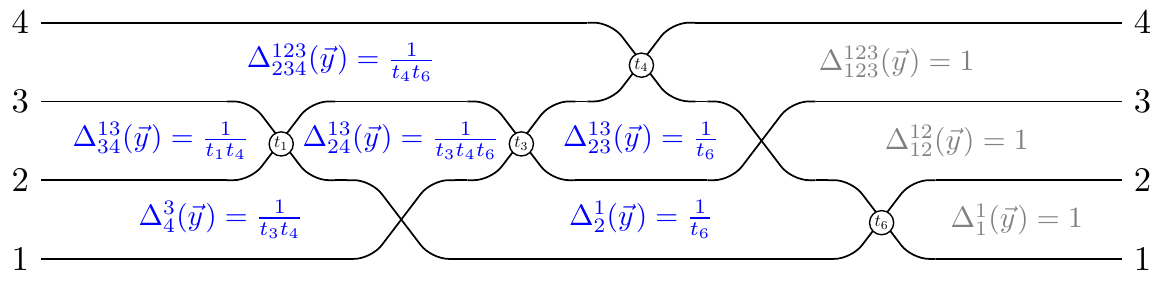}%
\\
(a) Right twist Chamber Ansatz: $\fR_j(\yR)=\chmnrR_{\RegL_r}(\yR)=\Delta^{\DRv(\RegL_r)}_{\DRw(\RegL_r)}(\yR)$.%
\\
\\
\includegraphics[width=0.98\textwidth]{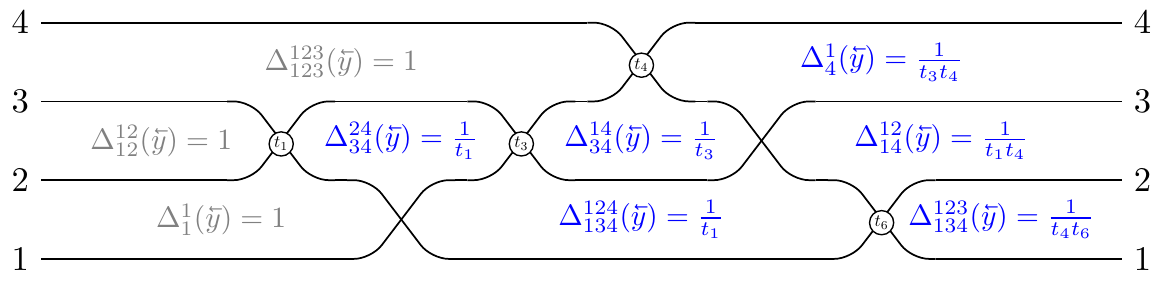}%
\\
(b) Left twist Chamber Ansatz: $\fL_j(\yL)=\chmnrL_{\RegR_r}(\yL)=\Delta^{\ULw(\RegR_r)}_{\ULv(\RegR_r)}(\yL)$.%
\end{tabular}
  \caption{\label{fig:ch_ans} Comparing the Chamber Ansatz formulas (\cref{thm:R_chamber_ans,thm:L_chamber_ans}). See also \cref{thm:A_ch_ans}.}
\end{figure}

 \begin{table}
\def\arrbig{1.8}
\def\arrmed{1.5}
\def\arrsmall{1.3}
\renewcommand{\arraystretch}{\arrbig}
\def\dst{}
\def\shl{\hspace{-0.07in}\renewcommand{\arraystretch}{\arrsmall}}
\begin{tabular}{|c|l|}\hline
\eqref{eq:Rich_dfn_R}, \eqref{eq:Rich_dfn_L} & $\Rich_v^w:=(B_-wB_-\cap BvB_-)/B_- $, \quad $\LRich_v^w:=B_-\bs(B_-wB_-\cap B_-vB)$\\
  \eqref{eq:NvN_to_Rich} & {\bf Chiral map $\chiv$:} $\dst \begin{tikzcd}[column sep=50pt,row sep=-10pt]%
\LRich_v^w \arrow[rr,"\chiv"',bend right=8] &  \capBWB(N\dv \cap \dv N) \arrow[l,"\sim"'] \arrow[r,"\sim"] & \Rich_v^w\end{tikzcd}$\\\hline
& {\bf Marsh--Rietsch parametrization and (pre-)twist maps:}\\
\eqref{eq:MR} & $g=\dst \gt  =  g_1\cdots g_m \in \capBWB(N\dv)$, \quad $\dst g_r=\ds_{i_r}$ if $\dst r\notin \Jo$,\quad $\dst g_r=x_{i_r}(t_r)$ if $\dst r\in \Jo$ \\
Thm.~\ref{thm:L_chamber_ans} & $\dst \gc\in \capBWB(N\dv\cap \dv N)$ \quad satisfies \quad $\dst gB_-=\gc B_-$\\
Dfn.~\ref{dfn:pre_twist} & \shl \renewcommand{\arraystretch}{\arrsmall}\begin{tabular}{l}
$\dst \yR=[g^T\dw]_L^+$,\quad $\dst \tpre_v^w(gB_-)= B_-v[g^T\dw]_L^+$,\quad $\dst \twist_v^w=\chiv\circ\tpre_v^w$\\
$\dst \yL=[\dw\gc^T]_R^+$,\quad $\dst \tpreL_v^w(B_-h)= [\dw h^T]_R^+vB_-$,\quad $\dst \twistL_v^w=\Ltpre_v^w \circ \chivi$ \\
                           \end{tabular}  \\
Thms.~\ref{thm:chiral_tnn},~\ref{thm:twist_tnn} & {\bf The maps $\chiv$, $\tpre_v^w$, and $\twist_v^w$ preserve total positivity}\\
\hline
& {\bf Chamber Ansatz formulas:}\\
\eqref{eq:fR_dfn}, \eqref{eq:fL_dfn} & $\dst \fR_r=\Delta^{\sv_{i_m}\cdots \sv_{i_r}\om_{i_r}}_{s_{i_m}\cdots s_{i_r}\om_{i_r}}$,  \quad $\dst \fL_r=\Delta^{s_{i_1}\cdots s_{i_r}w_0\om_{i_r^\ast}}_{\sv_{i_1}\cdots \sv_{i_r}w_0\om_{i_r^\ast}}$\\
Thm.~\ref{thm:R_chamber_ans} & Right twist Chamber Ansatz: express $t_r$ in terms of $\fR_j(\yR)$\\
Thm.~\ref{thm:L_chamber_ans} & Left twist Chamber Ansatz: express $t_r$ in terms of $\fL_j(\yL)$\\
  Thm.~\ref{thm:MR_chamber_ans} &\shl \renewcommand{\arraystretch}{1}\begin{tabular}{l}
Marsh--Rietsch Chamber Ansatz: express $t_r$ in terms of $\fL_j(z^T)$, \\
where $z\in N_-$ satisfies $zwB_-=gB_-$
                                      \end{tabular} \\
\hline
& {\bf Positivity-preserving involutions:}\\
\eqref{eq:inv_T} & $x\mapsto x^T$: anti-automorphism,  $x_i(t) \xleftrightarrow{T} y_i(t)$, $\ds_i^T=\ds_i^{-1}$; $a^T=a$  ($a\in\H$)\\
\eqref{eq:inv_iota} & $x\mapsto x^\iota$: anti-automorphism, preserves $x_i(t)$, $y_i(t)$, $\ds_i$; $a^\iota=a^{-1}$  ($a\in\H$)\\
  \eqref{eq:inv_theta} & $x\mapsto x^\theta=(x^\iota)^T$: automorphism, $x_i(t) \xleftrightarrow{\theta} y_i(t)$, $\ds_i^\theta=\ds_i^{-1}$; $a^\theta=a^{-1}$ ($a\in\H$)\\\hline
& {\bf Generalized minors:}\\
\eqref{eq:line_dfn} & $\lline{w}=\dot w=\ds_{i_1}\cdots \ds_{i_m},\quad \line w=\ds^{-1}_{i_1}\cdots \ds^{-1}_{i_m}$ \\
\eqref{eq:dom_dfn}, \eqref{eq:dom_iota} & $\DOM^v_w(x)=\dom \left(\lline{v^{-1}} x\line{w}\right)=\Delta^{ww_0\om_{i^\ast}}_{vw_0\om_{i^\ast}}(x^\iota)$\\\hline
& {\bf Collision moves:}\\
\shl\renewcommand{\arraystretch}{\arrmed}\begin{tabular}{c}
\\
\eqref{eq:collision_x}\\
\eqref{eq:collision_y}
\end{tabular} & \shl\renewcommand{\arraystretch}{\arrmed}\begin{tabular}{c}
 $t_+=1/t$,\quad $a_+=t^{\alphacheck_i}$, \quad $t_-=-1/t$\\
 $\ds_i^{-1} x_i(t)=x_i(t_-)a_+^{-1}y_i(t_+),\quad x_i(t)\ds_i^{-1}=y_i(t_+)a_+x_i(t_-)$\\
 $\ds_iy_i(t)=y_i(t_-)a_+x_i(t_+),\quad y_i(t)\ds_i=x_i(t_+)a_+^{-1}y_i(t_-)$
                \end{tabular}\\\hline
\eqref{eq:YpYoYm_dfn}, \eqref{eq:z_wyw} & $g^T\dw=\Yp\Yo \Ym$, \quad  $\yR=\Yp$, \quad $z=(\dw \Ym \dw^{-1})^T$ \\\hline
\end{tabular}
\renewcommand{\arraystretch}{1}
  \caption{\label{tab:formulas} Useful notation and formulas.}
\end{table}

\newpage

\section{Background}\label{sec:background}

\subsection{Pinnings}\label{sec:pinnings}
Our exposition mostly follows that of~\cite{GKL3}. We denote by $X(\H):=\Hom(\H,\Cast)$ the \emph{weight lattice} and for $\gamma\in X(\H)$ and $a\in\H$, we let $a^\gamma$ denote the value of $\gamma$ on~$a$. Recall that we have the root system $\Phi\subset X(\H)$ of $G$ and a set of simple roots $\{\alpha_i\}_{i\in I}$. Let $Y(\H):=\Hom(\Cast,\H)$ be the \emph{coweight lattice}. We have simple coroots $\{\alphacheck_i\}_{i\in I}\subset Y(\H)$ given by $t^{\alphacheck_i}:=\phi_i\begin{pmatrix}
t&0\\0&t^{-1}
\end{pmatrix}\in Y(\H)$. Consider the natural pairing $\<\cdot,\cdot\>: X(\H)\times Y(\H)\to \Z$ such that for all $\gamma\in X(\H)$, $\beta\in Y(\H)$, and $t\in \Cast$, we have $(t^\beta)^\gamma=t^{\<\gamma,\beta\>}$. Let $\{\omega_i\}_{i\in I}\subset X(\H)$ be the \emph{fundamental weights}, defined so that $\<\omega_i,\alphacheck_j\>=\delta_{i,j}$ for $i,j\in I$.

For $\gamma\in X(\H)$, $t\in \Cast$, $a\in \H$, and $w\in W$, we have~\cite[Equations~(1.2) and~(2.5)]{FZ} 
\begin{equation}\label{eq:torus_conj}
(\dw a\dw^{-1})^\gamma=a^{w\gamma},
\quad ax_i(t)a^{-1}=x_i(a^{\alpha_i}t),\quad ay_i(t)a^{-1}=y_i(a^{-\alpha_i}t).
\end{equation}
\noindent The totally positive part $\Htp$ of $\H$ is the subgroup generated by $t^{\alphacheck_i}$ for $i\in I$ and $t\in\R_{>0}$.

\subsection{Involutions}\label{sec:inv}
We will make use of three involutions on $G$; cf.~\cite[Section~2.1]{FZ_double}. We have two involutive anti-automorphisms $x\mapsto x^T$ and $x\mapsto x^\iota$ of $G$ and one involutive automorphism $x\mapsto x^\th=(x^T)^\iota=(x^\iota)^T$ defined by
\begin{align}%
\label{eq:inv_T}  a^T&=a \quad (a\in \H), &x_i(t)^T&=y_i(t), &y_i(t)^T&=x_i(t);\\
\label{eq:inv_iota}  a^\iota&=a^{-1} \quad (a\in \H), &x_i(t)^\iota&=x_i(t), &y_i(t)^\iota&=y_i(t);\\
\label{eq:inv_theta}  a^\th&=a^{-1} \quad (a\in \H), &x_i(t)^\th&=y_i(t), &y_i(t)^\th&=x_i(t).
\end{align}
\begin{remark}\label{rmk:inv_TNN}
Each of these three involutions clearly preserves the subset $\Gtnn\subset G$. 
 Thus we have involutive homeomorphisms
\begin{align}
\label{eq:iota_TNN} \GBtnn &\xrasim \BGtnn, &gB_-&\mapsto B_-g^\iota;\\
\label{eq:theta_TNN} \GBptnn &\xrasim \GBtnn, &gB&\mapsto g^\theta B_-.
\end{align}
We always apply the map~\eqref{eq:theta_TNN} to the objects in~\cite{MR,Ing} in order to switch from $G/B$ to $G/B_-$.
\end{remark}

The above involutions commute pairwise and they also commute with $x\mapsto x^{-1}$. We have
\begin{equation*}%
\ds_i^\iota=\ds_i=\phi_i\begin{pmatrix}
0 & 1 \\ -1 & 0
\end{pmatrix} \quad\text{and}\quad  \ds_i^T=\ds_i^\theta=\ds_i^{-1}=\phi_i\begin{pmatrix}
0 & -1 \\ 1 & 0
\end{pmatrix}.
\end{equation*}
For $w\in W$, choose a reduced expression $\bw=(i_1,\dots,i_m)$ and let
\begin{equation}\label{eq:line_dfn}
  \lline{w}:=\dot w=\ds_{i_1}\cdots \ds_{i_m},\quad \line w:=((w^{-1})^{\bigcdot})^{-1}=\ds^{-1}_{i_1}\cdots \ds^{-1}_{i_m}.
\end{equation}
Using this notation, we find~\cite[Proposition~2.1]{FZ_double}
\begin{equation*}%
  \dw^\iota=\lline{w^{-1}}, \quad \dw^T=\dw^{-1}=\line{w^{-1}}, \quad\text{and}\quad
\dw^\theta=\line{w}.
\end{equation*}

\subsection{Generalized minors}\label{sec:backgr_gen_mnrs}
The following construction appears in~\cite[Section~1.4]{FZ_double}. Recall that we set $\Gopm:=BB_-$ and $\Gomp:=B_-B$. For $i\in I$, let $\dom$ be the regular function whose restriction to $x\in \Gomp$ is given by
\begin{equation*}%
  \dom(x):=([x]_0^\mp)^{\om_i}.
\end{equation*}
In particular, for $x\in G$, $y_-\in N_-$, and $y_+\in N$, we have
\begin{equation}\label{eq:dom_Nmp}
  \dom(y_-xy_+)=\dom(x).
\end{equation}
For $i\in I$, $v,w\in W$, and $x\in G$, we set
\begin{equation}\label{eq:dom_dfn}
  \DOM^v_w(x):=\dom \left(\lline{v^{-1}} x\line{w}\right).
\end{equation}
For $a\in \H$, we have~\cite[Equation~(2.14)]{FZ_double}
\begin{equation}\label{eq:dom_H}
  \DOM^v_w(a\cdot x)=\DOM^v_v(a)\cdot \DOM^v_w(x).
\end{equation}
We also have the following identity~\cite[Equation~(2.25)]{FZ_double}:
\begin{equation}\label{eq:dom_iota}
  \DOM^v_w(x)= \DOM^w_v(x^T)=\Delta^{ww_0\om_{i^\ast}}_{vw_0\om_{i^\ast}}(x^\iota).
\end{equation}

\subsection{Factorizations}

For $u\in W$, we define two subgroups
\begin{equation}\label{eq:NN'_dfn}
  N(u):=N\cap \du^{-1} N_-\du \quad\text{and}\quad  N'(u):=N\cap \du^{-1} N\du.
\end{equation}
These subgroups do not depend on the choice of the representative $\du$ of $u$ inside $N_G(\H)$. The following result is well known.
\begin{lemma}\label{lemma:factor_NN'}
The multiplication map provides isomorphisms
\begin{equation*}%
  N(u)\times N'(u)\xrasim N, \quad N'(u)\times N(u)\xrasim N, \quad\text{and}\quad (\du^{-1}N\du\cap N)\times (\du^{-1} N\du\cap N_-)\xrasim \du^{-1}N\du.
\end{equation*}
\end{lemma}

\begin{lemma}[\cite{Lec}]\label{lemma:Lec_NN'}
Let $v\leq w\in W$ and let $(\bv,\bw)\in\Red(v,w)$. The functions $\fR_r$ defined in~\eqref{eq:fR_dfn} belong to the doubly-invariant ring $^{N(v)}\mathbb{C}[N]^{N'(w)}$. In other words, for all $(y,y')\in N(v)\times N'(w)$ and $x\in N$, we have
\begin{equation*}%
  \fR_r(yxy')=\fR_r(x).
\end{equation*}
\end{lemma}

Our next result gives an explicit procedure for converting an MR-parametrized element $g\in N\dv$ such that $gB_-\in\Rich_v^w$ into an element $\gc\in N\dv\cap \dv N$ satisfying $gB_-=\gc B_-$; cf.~\eqref{eq:NvN_to_Rich}.
\begin{lemma}\label{lemma:Xggc}
For any $g\in N\dv$, there exists a unique element $x\in N_-\cap \dv^{-1} N \dv$ such that 
\begin{equation*}%
  \gc:=gx\in N\dv\cap \dv N.
\end{equation*}
\end{lemma}
\begin{proof}
Using \cref{lemma:factor_NN'}, we factorize $\dv^{-1}g\in \dv^{-1}N\dv$ as a product $y_1y_2$ with $y_1\in \dv^{-1}N\dv \cap N$ and $y_2\in \dv^{-1}N\dv \cap N_-$. Setting $\gc:=\dv y_1$ and $x:=y_2^{-1}$ shows the existence part. Since the procedure can be reversed, the uniqueness part follows.
\end{proof}

The next lemma is well known; see e.g.~\cite[Equation~(4.33)]{GKL3}.
\begin{lemma}
Let $v, w\in W$. We have the following isomorphisms:
\begin{align}%
\label{eq:dec_w_R}  N_-\dw\cap  \dw N &\xrasim (B_-wB_-)/B_-, & g&\mapsto  g B_-;\\
\label{eq:dec_w_L}  \dw N_-\cap N\dw &\xrasim \BMX(B_-wB_-), & g&\mapsto B_- g ;
\end{align}
\end{lemma}

\begin{lemma}
For $w\in W$, $b\in B$, and $b_-\in B_-$, we have $\dw b \dw^{-1}\in \Gomp$, $\dw^{-1}b_-\dw\in\Gopm$, 
\begin{equation}\label{eq:Gauss_wBwi}
  [\dw b \dw^{-1}]_R^+\in \dw N \dw^{-1}\cap N, \quad\text{and}\quad  [\dw^{-1} b_- \dw]_L^+\in \dw^{-1} N_- \dw \cap N.
\end{equation}
In particular, for $g\in B_-wB_-$, we have
\begin{equation}  \label{eq:Gauss_dwi_g} 
 \dw^{-1}g\in \Gopm,  \quad\text{and}\quad [\dw^{-1}g]_L^+\in \dw^{-1}N_-\dw\cap N;
\end{equation}
\end{lemma}
\begin{proof}
Let $b=ay$ with $a\in\H$ and $y\in N$. Using \cref{lemma:factor_NN'}, we factorize $y=y_1y_2$ with $y_1\in N\cap \dw^{-1} N_-\dw$ and $y_2\in N\cap \dw^{-1} N\dw$. Then
\begin{equation*}%
  \dw b\dw^{-1}=\dw ay_1\dw^{-1}\cdot \dw y_2\dw^{-1} \in B_-\cdot (\dw N\dw^{-1}\cap N).
\end{equation*}
This shows the first part of~\eqref{eq:Gauss_wBwi}, and the second part follows by a similar argument.

To show~\eqref{eq:Gauss_dwi_g}, let $b_-\in B_-$ be such that $g\in b_- wB_-$. By~\eqref{eq:Gauss_wBwi}, we have $\dw^{-1}b_-\dw\in (\dw^{-1} N_- \dw \cap N)\cdot B_-$, and thus 
\begin{equation*}%
  \dw^{-1}g\in \dw^{-1}b_-\dw B_-\subset (\dw^{-1} N_- \dw \cap N)\cdot B_-.\qedhere
\end{equation*}
\end{proof}

\subsection{Collision moves}
By~\cite[Equation~(2.13)]{FZ}, for each $t\in \Cast$ and $i\in I$, there exist  $t_+=1/t$, $a_+=t^{\alphacheck_i}\in \H$, and $t_-=-1/t$ satisfying 
\begin{equation}\label{eq:collision_x}
\ds_i^{-1} x_i(t)=x_i(t_-)a_+^{-1}y_i(t_+),\quad x_i(t)\ds_i^{-1}=y_i(t_+)a_+x_i(t_-),
\end{equation}
\begin{equation}\label{eq:collision_y}
\ds_iy_i(t)=y_i(t_-)a_+x_i(t_+),\quad y_i(t)\ds_i=x_i(t_+)a_+^{-1}y_i(t_-).
\end{equation}
Note that if $t >0$ then $t_+ >0$, $a_+\in\Htp$, and $t_- < 0$. 

We also record the following \emph{conjugation moves}.
\begin{lemma}\label{lemma:conj}
Let $w\in W$, $i\in I$, and $t\in\Cast$.
\begin{theoremlist}[label=\normalfont(\roman*)]
\item\label{conj1} If $ws_i<w$ then $\dw x_i(t)\dw^{-1}\in N_-$ and $\dw y_i(t)\dw^{-1}\in N$.
\item\label{conj2} If $ws_i>w$ then $\dw x_i(t)\dw^{-1}\in N$ and $\dw y_i(t)\dw^{-1}\in N_-$.
\item\label{conj3} For all $h\in N'(s_i)$, we have $y_i(t)hy_i(-t)\in N$.
\item\label{conj4} For all $h\in N$, we have $\ds_i^{-1} h \ds_i=h' y_i(t')$ for some $h'\in N$ and $t'\in \C$. 
\item\label{conj5} We have $\dw_0 x_{i^\ast}(-t) \dw_0^{-1}=y_{i}(t)$.
\end{theoremlist}
\end{lemma}
\begin{proof}
Parts~\itemref{conj1} and~\itemref{conj2} are well known and may be deduced e.g. from~\cite[Section~4.2]{GKL3}. Part~\itemref{conj3} is also well known and follows from~\cite[Lemma~4.3]{GKL3}. For part~\itemref{conj4}, we use \cref{lemma:factor_NN'} to factorize $h=h_1h_2$ with $h_1\in N\cap \ds_i N\ds_i^{-1}$ and $h_2\in N\cap \ds_i N_-\ds_i^{-1}$. Then $h':=\ds_i^{-1}h_1\ds_i\in N$ and $\ds_i^{-1}h_2\ds_i=y_i(t')$, as expected.

For part~\itemref{conj5}, it is clear that we have $\dw_0 x_{i^\ast}(-t) \dw_0^{-1}=y_{i}(t')$ for some $t'\in \C$. Applying~\eqref{eq:dom_dfn}, we get
\begin{equation*}%
  \Delta^{s_{i}\om_{i}}_{\om_{i}}(\dw_0 x_{i^\ast}(-t) \dw_0^{-1})=
\Delta^{\om_{i}}_{\om_{i}}(\ds_i\dw_0 x_{i^\ast}(-t) \dw_0^{-1})=
\Delta^{\om_{i}}_{\om_{i}}(\dw_0\ds_{i^\ast} x_{i^\ast}(-t) \dw_0^{-1}).
\end{equation*}
Taking the inverse of the second equation in~\eqref{eq:collision_x}, we find $\ds_{i^\ast} x_{i^\ast}(-t)=x_{i^\ast}(t_+) a_+^{-1}y_{i_\ast}(t_-)$, where $a_+=t^{\alphacheck_{i^\ast}}$, $t_+=1/t$, and $t_-=-1/t$. Applying part~\itemref{conj2}, we find that 
\begin{equation*}%
  \dw_0x_{i^\ast}(t_+) a_+^{-1}y_{i_\ast}(t_-)\dw_0^{-1}\in N_-\dw_0 a_+^{-1} \dw_0^{-1} N.
\end{equation*}
 By~\eqref{eq:torus_conj} and~\eqref{eq:dom_Nmp}, we get
\begin{equation*}%
\Delta^{\om_{i}}_{\om_{i}}(\dw_0\ds_{i^\ast} x_{i^\ast}(-t) \dw_0^{-1})=\dom(\dw_0 a_+^{-1} \dw_0^{-1})=(\dw_0 a_+^{-1} \dw_0^{-1})^{\om_i}=t^{\<\om_{i^\ast},\alphacheck_{i^\ast}\>}=t.
\end{equation*}
On the other hand, by~\eqref{eq:collision_y} and~\eqref{eq:dom_Nmp},
\begin{equation*}%
  \Delta^{s_{i}\om_{i}}_{\om_{i}}(y_{i}(t))= \dom(\ds_i y_i(t))=\dom(y_i(t_-) t^{\alphacheck_i} x_i(t_+))=\dom(t^{\alphacheck_i})=t^{\<\om_{i},\alphacheck_{i}\>}=t.
\end{equation*}
This finishes the proof of part~\itemref{conj5}.
\end{proof}

\section{Basic properties of the twist}\label{sec:basic}
The goal of this section is to prove that $\twist_v^w$ gives a well-defined automorphism of $\Rich_v^w$; see \cref{cor:twist_well_def}. Since $\twist_v^w$ is defined as the composition of the chiral map $\chiv$ and the pre-twist $\tpre_v^w$, we need to study these maps first. Recall that the pre-twist is induced by the map $g\mapsto \dv[g^T\dw]_L^+$ introduced in \cref{prop:pre_pre_twist}.
\begin{proof}[Proof of \cref{prop:pre_pre_twist}]
  Let $g\in\capBWB(N\dv)$. Since $g\in B_-wB_-$, by~\eqref{eq:Gauss_dwi_g}, we have $\dw^{-1}g\in \Gopm$, and thus $g^T\dw\in\Gopm$. In particular, $[g^T\dw]_L^+$ is well defined. Set $h:=\dv[g^T\dw]_L^+ \in\dv N$. By definition, $[g^T\dw]_L^+\in g^T\dw \cdot B_-$. Since $g\in N\dv$, we find
\begin{equation*}%
  h=\dv[g^T\dw]_L^+\in \dv g^T\dw \cdot B_-\subset \dv(\dv^{-1}N_-)\dw\cdot B_-\subset B_-wB_-.
\end{equation*}
We have shown that $h\in \capBWB(\dv N)$. A similar argument shows that $[\dw h^T]_R^+$ is well defined and that $[\dw h^T]_R^+\dv\in \capBWB(N\dv)$.  It remains to show that the two maps are inverses of each other, i.e., that $g=[\dw h^T]_R^+\dv$.

Write
\begin{equation*}%
  \dw^{-1}g=g_+g_0g_-,\quad\text{where}\quad (g_+,g_0,g_-)\in N\times \H\times N_-.
\end{equation*}
By~\eqref{eq:dec_w_R}, we have $g_+\in \dw^{-1}N_-\dw\cap N$. By definition, we have $h=\dv g_-^T$. We find

\begin{equation*}%
  \dw h^T=\dw g_-\dv^{-1}=\dw (g_0^{-1}g_+^{-1}\dw^{-1}g)  \dv^{-1}=(\dw g_0^{-1}\dw^{-1}) \cdot (\dw g_+^{-1}\dw^{-1})\cdot (g  \dv^{-1}).
\end{equation*}
Since $g_0\in \H$, we have $\dw g_0^{-1}\dw^{-1}\in \H$. Since $g_+\in \dw^{-1}N_-\dw\cap N$, we have $\dw g_+^{-1}\dw^{-1}\in N_-$.  Thus, $\dw h^T\in B_-\cdot (g\dv^{-1})$. Since $g\in N\dv$, we have $[\dw h^T]_R^+=g\dv^{-1}$, so indeed $[\dw h^T]_R^+\dv=g$.
\end{proof}

Our next result will be used to show that the right pre-twist  $\tpre_v^w$ is well defined.
\begin{lemma}\label{lemma:pre_twist_well_def}
Let $g_1,g_2\in \capBWB(N\dv)$ be such that $g_1B_-=g_2B_-$. Let $h_1:=\dv[g_1^T\dw]_L^+$ and $h_2:=\dv[g_2^T\dw]_L^+$. Then $B_-h_1=B_-h_2$.
\end{lemma}
\begin{proof}
Since $g_1,g_2\in N\dv$ and $g_1B_-=g_2B_-$, by \cref{lemma:Xggc}, we have $g_2=g_1x$ for some $x\in \dv^{-1}N\dv\cap N_-$. Since $x^T\in N$, we have
\begin{equation*}%
  h_2=\dv[g_2^T\dw]_L^+=\dv[x^Tg_1^T\dw]_L^+=\dv x^T[g_1^T\dw]_L^+=\dv x^T\dv^{-1} \cdot \dv[g_1^T\dw]_L^+=\dv x^T\dv^{-1}\cdot h_1.
\end{equation*}
Since $x^T\in \dv^{-1}N_-\dv$, we find $\dv x^T\dv^{-1}\in N_-$, and thus $h_2\in B_-h_1$.
\end{proof}

\begin{corollary}
The right and left pre-twists $\tpre_v^w$ and $\tpreL_v^w$ are well-defined mutually inverse isomorphisms between $\Rich_v^w$ and $\LRich_v^w$.
\end{corollary}
\begin{proof}
Each point of $\Rich_v^w$ is of the form $gB_-$ for some $g\in\capBWB(N\dv)$ in view of~\eqref{eq:NvN_to_Rich}. However, the projection map $\capBWB(N\dv)\to\Rich_v^w$, $g\mapsto gB_-$ is surjective but in general not injective. \Cref{lemma:pre_twist_well_def} shows that $\tpre_v^w(gB_-)$ does not depend on the choice of the representative $g\in\capBWB(N\dv)$, and thus $\tpre_v^w$ is well defined on $\Rich_v^w$. By \cref{prop:pre_pre_twist}, its image is contained inside $\LRich_v^w$.  A similar argument shows that the left pre-twist $\tpreL_v^w$ is well defined on $\LRich_v^w$ and has image contained inside $\Rich_v^w$. It follows from \cref{prop:pre_pre_twist} that the two maps are mutually inverse isomorphisms.
\end{proof}

It follows from~\eqref{eq:NvN_to_Rich} that the chiral map $\chiv$ restricts to an isomorphism $\RichL_v^w\xrasim\Rich_v^w$, and thus we get the following result.
\begin{corollary}\label{cor:twist_well_def}
The right and left twist maps $\twist_v^w$ and $\twistL_v^w$ are well-defined mutually inverse automorphisms of $\Rich_v^w$.
\end{corollary}

\section{Marsh--Rietsch Chamber Ansatz}\label{sec:MR_chamber_ansatz}
Our goal is to recall the results of~\cite{MR} and connect them to our Chamber Ansatz formulas. The following result appears in~\cite[Theorem~7.1(1)]{MR}.  We shall compare it to \cref{thm:L_chamber_ans}.

\begin{theorem}[Marsh--Rietsch Chamber Ansatz~\cite{MR}]\label{thm:MR_chamber_ans}
Let $v\leq w$ in $W$ and $(\bv,\bw)\in\Red(v,w)$. Let $g:=\gt\in\capBWB(N\dv)$ be MR-parametrized, and let $z\in N_-$ be any element satisfying $gB_-=zw B_-$. Then for all $r\in\Jo$, we have
\begin{equation}\label{eq:MR_chamber_ans}
  t_r=\frac{ \prod_{j\neq i_r^\ast} \Delta^{\wi r w_0\om_j}_{\vi rw_0\om_j}(z^T)^{-a_{j,i_r^\ast}}  }
{\Delta^{\wi r w_0\om_{i_r^\ast}}_{\vi r w_0\om_{i_r^\ast}}(z^T) \Delta^{\wi{r-1}w_0\om_{i_r^\ast}}_{\vi{r-1}w_0\om_{i_r^\ast}}(z^T)   }.
\end{equation}
Conversely, for each $j\in[m]$, we have 
\begin{equation}\label{eq:MR_chamber_ans_inv}
  \fL_j(z^T)=\prod_{r\in\Jo:\, r\leq j } t_r^{-\<\om_{i_j},s_{i_j}s_{i_{j-1}}\cdots s_{i_{r+1}} \ach_{i_r}\>}.
\end{equation}
\end{theorem}

\begin{proof}[Proof of \cref{thm:L_chamber_ans}]
Let $\gc\in \capBWB(N\dv\cap\dv N)$ and $\yL:=[\dw \gc^T]_R^+$ be as in \cref{thm:L_chamber_ans}. 
By~\cite[Theorem~7.1]{MR}, the right-hand side of~\eqref{eq:MR_chamber_ans} does not depend on the choice of the element $z\in N_-$ satisfying $gB_-=zwB_-$. Our goal is to show that the element $z:=\yLT$ satisfies these assumptions. Clearly, $\yL\in N$, so $z\in N_-$. We concentrate on showing that $gB_-=\yLT wB_-$.

Since $\gc\in B_-wB_-$, by~\eqref{eq:Gauss_dwi_g}, we have $\dw^{-1}\gc\in\Gopm$, and thus $\gc^T\dw\in\Gopm$. Moreover,~\eqref{eq:Gauss_dwi_g} implies that
\begin{equation}\label{eq:gcT_dw_in_B_Ym}
  \gc^T\dw=b \Ym \quad\text{for some $b\in B$ and $\Ym\in \dw^{-1}N\dw\cap N_-$.}
\end{equation}

By definition, $\gc B_-=gB_-$. Thus
\begin{equation}\label{eq:BgT=BYmdwi}
  Bg^T=B\gc^T=B\Ym\dw^{-1}.
\end{equation}
On the other hand, conjugating~\eqref{eq:gcT_dw_in_B_Ym} by $\dw$ and using $\dw\Ym\dw^{-1}\in N$, we find
\begin{equation*}%
  \dw\gc^T= \dw b\dw^{-1}\cdot \dw \Ym\dw^{-1},\quad \yL=[\dw\gc^T]_R^+=[\dw b\dw^{-1}]_R^+\cdot \dw \Ym\dw^{-1}.
\end{equation*}
 Let $b':=[\dw b\dw^{-1}]_R^+$. By~\eqref{eq:Gauss_wBwi}, we have $b'\in \dw N \dw^{-1}\cap N$.  In particular, $Bw^{-1} b'=Bw^{-1}$. Therefore
\begin{equation}\label{eq:BYmdwi=BwiyL}
  B\Ym \dw^{-1}=Bw^{-1}\cdot \dw \Ym\dw^{-1}=Bw^{-1}b'\cdot \dw\Ym\dw^{-1}=Bw^{-1}\yL.
\end{equation}
Combining~\eqref{eq:BgT=BYmdwi} with~\eqref{eq:BYmdwi=BwiyL}, we get $Bg^T=Bw^{-1}\yL$, and thus $gB_-=\yLT wB_-$.
\end{proof}

\section{Chiral and reversal maps}\label{sec:chiral}
The goal of \cref{sec:chiral_tnn_proof} is to prove \cref{thm:chiral_tnn}. In \cref{sec:reversal_maps}, we study \emph{reversal maps} which swap the roles of $(v,w)$ and $(ww_0,vw_0)$. In \cref{sec:chiral_op}, we use them to introduce the \emph{opposite chiral map} which will be later used in \cref{sec:twist-TNN} to show that the twist map preserves total positivity.

\subsection{The chiral map preserves total positivity}\label{sec:chiral_tnn_proof}
\begin{lemma}\label{lemma:y_i(t)gB_-_in_Rtp}
Let $v\leq w$ in $W$ and $gB_-\in \Rtp_v^w$. Let $i\in I$ be such that $s_iv<v$. Then 
\begin{equation*}%
  y_i(t)gB_-\in \Rtp_{s_iv}^w \quad\text{for all $t>0$.}
\end{equation*}
\end{lemma}
\begin{proof}
For $t>0$, we have $y_i(t)\in\Gtnn$, and thus $y_i(t) \cdot  \GBtnn\subset \GBtnn$; cf.~\cite{Lus2}. Since $y_i(t)\in B_-$ and $g\in B_-wB_-$, we have $y_i(t)g\in B_-wB_-$. Since $y_i(t)\in Bs_iB$ and $g\in BvB_-$, by the properties of the Bruhat decomposition, we have $y_i(t)g\in Bs_ivB_-$ when $s_iv<v$. By definition, $\Rtp_{s_iv}^w=\GBtnn\cap Bs_ivB_-\cap B_-wB_-$.
\end{proof}

We will need the following technical result.
\begin{lemma}\label{lemma:dudv}
Let $v,u_1,u_2\in W$ be such that $\ell(u_1v)=\ell(u_1)+\ell(v)$ and $\ell(vu_2)=\ell(v)+\ell(u_2)$. Let $w:=vu_2$ and let $\bw$ be a reduced expression for $w$ that starts with a reduced expression for $v$. Let $\bv$ be the positive distinguished subexpression for $v$ inside $\bw$.\footnote{Note that $\bv$ is the \emph{rightmost} subexpression for $v$ inside $\bw$, while the start of $\bw$ gives the \emph{leftmost} subexpression for $v$, and the two subexpressions may overlap.} Choose an MR-parametrized element $g=\gt\in \capBWB(N\dv)$ with $\t \in (\R_{>0})^{\Jo}$ and set
\begin{equation*}%
  \cev g:=\du_1\dv[\dv^{-1}g]_L^+.
\end{equation*}
Then $B_-\cev g\in \BGtnn$.
\end{lemma}
\begin{proof}
We prove the result by induction on $\ell(v)$. 

Consider the base case $v=\id$. Then $\dv^{-1}g\in N$ so $\cev g=\du_1 g$.  The element $g$ is a product of elements of the form $x_i(t)\in\Gtnn$ whose right action preserves $\BGtnn$, and thus indeed  $B_-\cev g=B_-\du_1 g\in \BGtnn$ since $B_-\du_1\in\BGtnn$.

For the induction step, assume that $\ell(v)>0$ and that the result has been shown for all $v'$ satisfying $\ell(v')<\ell(v)$. Write $g=g_1g_2\cdots g_m$ as in~\eqref{eq:MR}. Suppose that the reduced expressions for $v$ and $w$ start with $s_i$. We consider two cases: either $g_1=\ds_i$ or $g_1=x_i(t)$.  In each case, we will set $g':=g_2\cdots g_m$, $v':=s_iv<v$, $u_1':=u_1s_i$, and $u_2':=u_2$.

Assume first that $g_1=\ds_i$ for $i\in I$. Then $(\dv')^{-1} g'=\dv g$ and $\du'_1\dv'=\du_1\dv$, so 
\begin{equation*}%
  \cev g\,':=\du'_1\dv'[(\dv')^{-1} g']_L^+=\du_1\dv[\dv^{-1}g]_L^+=\cev g.
\end{equation*}
By the induction hypothesis, we have $B_-\cev g\in \BGtnn$.

Assume now that $g_1=x_i(t)$, with $t>0$.  Then $\dv^{-1}$ ends with $\ds_i^{-1}$.  Applying a collision move~\eqref{eq:collision_x}, we replace $\ds_i^{-1} x_i(t)$ with $x_i(t_-)a_+y_i(t_+)$ and get
\begin{equation}\label{eq:dudv_etc}
 \du_1\dv [\dv^{-1} g]_L^+=\du_1 \ds_i\dv'[(\dv')^{-1}\cdot  x_i(t_-)a_+y_i(t_+) \cdot g']_L^+.
\end{equation}
First, we have $g'B_-\in\Rtp_{v}^{s_iw}$, and thus by \cref{lemma:y_i(t)gB_-_in_Rtp}, $y_i(t_+) g'B_-\in\Rtp_{s_iv}^{s_iw}$. Therefore $a_+ y_i(t_+) g'B_-\in\Rtp_{s_iv}^{s_iw}$. Recall that we set $v':=s_iv$ and $u_2':=u_2$, and thus $w':=v'u_2'=s_iw$. Choose a reduced expression $\bw'$ for $w'$ starting with a reduced expression for $v'$ and let $g''=\gbf_{\bv',\bw'}(\bt')$ be the corresponding MR-parametrized element satisfying $a_+y_i(t_+)g'B_-=g''B_-$. Since $s_iv'>v'$, by \cref{conj2}, we find $(\dv')^{-1}x_i(t_-) \dv' \in N$. Combining the pieces together, the right-hand side of~\eqref{eq:dudv_etc} transforms into
\begin{equation*}%
\du_1 \ds_i\dv'[(\dv')^{-1}x_i(t_-)\dv'\cdot  (\dv')^{-1} \cdot g'']_L^+= \du_1 \ds_i\dv'\cdot (\dv')^{-1}x_i(t_-)\dv'[  (\dv')^{-1} \cdot g'']_L^+=\du_1 \ds_ix_i(t_-)\dv'[(\dv')^{-1} \cdot g'']_L^+.
\end{equation*}
Since $\ell(u_1v)=\ell(u_1)+\ell(v)$ and $s_iv<v$, we get $u_1s_i>u_1$, and thus $\du_1\ds_ix_i(t_-)\in N_-\du_1\ds_i$ by \cref{conj1}. Thus
\begin{equation}\label{eq:dudv_etc2}
  B_-\du_1 \ds_ix_i(t_-)\dv'[(\dv')^{-1} \cdot g'']_L^+=B_-\du_1 \ds_i\dv'[(\dv')^{-1} \cdot g'']_L^+=B_-\du_1'\dv'[(\dv')^{-1} \cdot g'']_L^+.
\end{equation}
By the induction hypothesis, the right-hand side of~\eqref{eq:dudv_etc2} belongs to $\BGtnn$.
\end{proof}

\begin{proof}[Proof of \cref{thm:chiral_tnn}]
We first prove the special case where $w=w_0$. Choose a reduced expression $\bw_0=(i_1,\dots,i_m)$ for $w_0$ that starts with a reduced expression for $v$. Let $\bv$ be the positive distinguished subexpression for $v$ inside $\bw_0$. Let $g:=\gbf_{\bv,\bw_0}(\bt)$ 
 be MR-parametrized. As we showed in the proof of \cref{lemma:Xggc}, the element $\gc\in \capBWoB(N\dv\cap \dv N)$ satisfying $\gc B_-=gB_-$ is given by
\begin{equation*}%
  \gc=\dv[\dv^{-1} g]_L^+.
\end{equation*}
Applying \cref{lemma:dudv} for $u_1=\id$ and $u_2=v^{-1}w_0$, we see that $B_-\gc\in\BGtnn$. We have shown that for $gB_-\in\Rtp_v^{w_0}$, the element $B_-\gc=\chivi(gB_-)$ belongs to $\BGtnn$. Clearly, $B_-\gc\in \LRtp_v^{w_0}$. The converse direction is handled similarly. We are done with the case $w=w_0$. 

Let $v\leq w$ in $W$. Let $\gc\in\capBWB(N\dv\cap \dv N)$ be such that $\gc B_-\in\Rtp_v^w$. Then $\gc B_-$ belongs to the closure of $\Rtp_v^{w_0}$ inside $\GBtnn$. The map $\chivi$ is continuous on the opposite Schubert cell $(BvB_-)/B_-$ which contains both $\Rtp_v^w$ and $\Rtp_v^{w_0}$. Therefore, $B_-\gc=\chivi(\gc B_-)$ belongs to the closure of $\chivi(\Rtp_v^{w_0})=\LRtp_v^{w_0}$. In particular, $B_-\gc\in \BGtnn$. 
\end{proof}

\subsection{Reversal maps}\label{sec:reversal_maps}
Throughout this section, we fix a pair $v\leq w$ in $W$ and set $\vv:=ww_0$, $\ww:=vw_0$. Our goal is to relate $\Rtp_v^w$ to $\Rtp_{\vv}^{\ww}$ and $\LRtp_v^w$ to $\LRtp_{\vv}^{\ww}$.  We will achieve this by introducing two maps: $g\mapsto g^\th\dw_0$ and $g\mapsto g^{-\iota}\dw_0$, where $g^{-\iota}:=(g^{-1})^\iota=(g^\iota)^{-1}$. 

The first map $g\mapsto g^\th\dw_0$ descends to a map on $G/B_-$: for $b\in B_-$, it sends $gb_-\mapsto g^\th b_-^\th\dw_0=g^\th \dw_0\cdot \dw_0^{-1}b_-^\th\dw_0$, where $b_-^\th\in B$ and thus $\dw_0^{-1}b_-^\th\dw_0\in B_-$. The fact that this map preserves total positivity was recently shown by Lusztig.

\begin{proposition}[{\cite[\S1.5]{Lusztig_Springer}}]\label{prop:rev1} 
The map $\revR:G/B_-\xrasim G/B_-$ sending $gB_-\mapsto g^\theta\dw_0B_-$ induces homeomorphisms
\begin{equation*}%
  (G/B_-)_{\geq0}\xrasim (G/B_-)_{\geq0} \quad\text{and}\quad \Rtp_v^w\xrasim \Rtp_{\vv}^{\ww}.
\end{equation*}
\end{proposition}

The second map $g\mapsto g^{-\iota}\dw_0$ descends to a map on $\BG$: for $b\in B_-$, it sends $b_-g\mapsto b_-^{-\iota}g^{-\iota} \dw_0$, where $b_-^{-\iota}\in B_-$. It turns out that this map also preserves total positivity.
\begin{proposition}\label{prop:rev2}
The map $\revL:\BG\xrasim \BG$ sending $B_-g\mapsto B_-g^{-\iota}\dw_0$ induces homeomorphisms
\begin{equation*}%
  \BGtnn\xrasim \BGtnn \quad\text{and}\quad \LRtp_v^w\xrasim \LRtp_{\vv}^{\ww}.
\end{equation*}
\end{proposition}
\begin{proof}
Recall from~\eqref{eq:iota_TNN} that the map $B_-g\mapsto g^\iota B_-$ yields an involutive isomorphism $\BGtnn\xrasim \GBtnn$. Letting $h:=g^\iota$, we need to show that the map $hB_-\mapsto \dw_0h^{-\iota}B_-$ induces a homeomorphism $\GBtnn\xrasim \GBtnn$. We may assume that $h=\gt=g_1g_2\cdots g_m$ is MR-parametrized as in~\eqref{eq:MR}. Write 
\begin{equation}\label{eq:rev2_iota}
  \dw_0h^{-\iota}B_-=\dw_0g_1^{-\iota}\dw_0^{-1}\cdot \dw_0g_2^{-\iota}\dw_0^{-1}\cdots \dw_0g_m^{-\iota}\dw_0^{-1}\cdot \dw_0B_-.
\end{equation}
If $g_1=x_i(t)$ then $g_1^{-\iota}=x_i(-t)$ and by \cref{conj5}, $\dw_0g_1^{-\iota}\dw_0^{-1}=y_{i^\ast}(t)$, which is equal to $x_{i^\ast}(t)^\th$. Set
\begin{equation}\label{eq:MR_ast}
  \gtast:=g_1^\ast g_2^\ast\cdots g_m^\ast, \quad\text{where} \quad g_r^\ast=
  \begin{cases}
    \ds_{i_r^\ast}, &\text{if $r\notin \Jo$,}\\
    x_{i^\ast}(t_r), & \text{if $r\in \Jo$.}
  \end{cases}
\end{equation}
Let $w^\ast:=w_0ww_0$. For $\bw=(i_1,\dots,i_m)$, let $\bw^\ast:=(i_1^\ast,\dots,i_m^\ast)$, and let $\bv^\ast$ be the positive distinguished subexpression for $v^\ast:=w_0vw_0$ inside $\bw^\ast$. Comparing~\eqref{eq:MR_ast} to~\eqref{eq:MR}, we see that $\gtast=\gbf_{\bv^\ast,\bw^\ast}(\bt)$, and thus $\gtast B_-\in \GBtnn$. In particular,~\eqref{eq:rev2_iota} implies that 
\begin{equation*}%
  \dw_0h^{-\iota}B_-=\gtast^\th \dw_0 B_-.
\end{equation*}
The result belongs to $\GBtnn$ by \cref{prop:rev1}.
\end{proof}

\subsection{The opposite chiral map}\label{sec:chiral_op}
The chiral map relates the opposite Schubert cells $B_-\bs (B_-vB)$ and $(BvB_-)/B_-$ for $v$. There is a similar map, relating the Schubert cells $B_-\bs (B_-wB_-)$ and $(B_-wB_-)/B_-$ for $w$. First, for $h\in G$, write $h^{-T}:=(h^{-1})^T=(h^T)^{-1}$. Recall from~\eqref{eq:dec_w_R} and~\eqref{eq:dec_w_L} that we have isomorphisms
\begin{align}%
\label{eq:N_-wN1} N_-\dw\cap \dw N&\xrasim (B_-wB_-)/B_-, & h&\mapsto hB_-;\\
\label{eq:N_-wN2} N_-\dw\cap \dw N&\xrasim B_-\bs(B_-wB_-), & h&\mapsto B_-h^{-T}.
\end{align}
\begin{definition}
Let $w\in W$. The \emph{opposite chiral map} 
\begin{equation*}%
  \chiw: B_-\bs(B_-wB_-)\xrasim (B_-wB_-)/B_-
\end{equation*}
is obtained by identifying both sides with $N_-\dw\cap \dw N$ via~\eqref{eq:N_-wN1}--\eqref{eq:N_-wN2}. We also consider restrictions $\chiw: \LRich_v^w\xrasim \Rich_v^w$ for each $v\leq w$. The inverse of $\chiw$ is denoted $\chiwi$.
\end{definition}

We have the following direct analog of \cref{thm:chiral_tnn}.
\begin{theorem}\label{thm:chiral_tnn_op}
Let $w\in W$. Then for any $h\in N_-\dw \cap \dw N$, we have
\begin{equation}\label{eq:chiral_tnn_op}
  hB_-\in \GBtnn \quad \Longleftrightarrow \quad B_-h^{-T}\in\BGtnn.
\end{equation}
In other words, for all $v\leq w \in W$, the map $\chiw: \LRich_v^w\xrasim \Rich_v^w$ restricts to a homeomorphism
\begin{equation*}%
  \chiw: \LRtp_v^w\xrasim \Rtp_v^w.
\end{equation*}
\end{theorem}
\begin{proof}
Similarly to \cref{sec:reversal_maps}, set $\vv:=ww_0$ and $\ww:=vw_0$.  We have the following commutative diagram consisting of four involutive isomorphisms:
\begin{equation}\label{eq:reversal_cd}
\begin{tikzcd}[column sep=20pt,row sep=20pt]
\capBWB(N\dv\cap \dv N) \arrow[rrr, "\gc\mapsto \gc^\theta \dw_0"] \arrow[d,"\gc\mapsto\gc"]
&&& N_-\dww\cap \dww N\cap B\dvv B_-  \arrow[d, "h\mapsto h^{-T}"] \\
\capBWB(N\dv\cap \dv N) \arrow[rrr, "\gc\mapsto \gc^{-\iota}\dw_0"] &&& N\dww\cap \dww N_-\cap B_-\dvv B.
\end{tikzcd}  
\end{equation}
We have already shown that the left vertical map and both horizontal maps preserve total positivity. More precisely, fix $\gc\in \capBWB(N\dv\cap \dv N)$. By \cref{prop:rev1}, $\gc B_-\in \GBtnn$ if and only if $\gc^\th\dw_0 B_-\in \GBtnn$. By \cref{thm:chiral_tnn}, $\gc B_-\in\GBtnn$ if and only if $B_-\gc\in \BGtnn$. By \cref{prop:rev2}, $B_-\gc\in \BGtnn$ if and only if $B_-\gc^{-\iota}\dw_0\in \BGtnn$. Setting $h:=\gc^\th\dw_0$ and composing the three maps, we see that $h\in \GBtnn$ if and only if $h^{-T}B_-\in \GBtnn$. Thus indeed the map $\chiw$ preserves total positivity, i.e.,~\eqref{eq:chiral_tnn_op} is satisfied for all $h\in N_-\dw\cap \dw N$. 
\end{proof}

\section{Twist maps and total positivity}\label{sec:twist-TNN}
In this section, we show that the twist and pre-twist maps preserve total positivity. We fix a pair $v\leq w$ in $W$, an MR-parametrized element $g=\gt \in \capBWB(N\dv)$ with $\bt\in(\R_{>0})^{\Jo}$, and the corresponding element $\gc\in \capBWB(N\dv\cap\dv N)$ satisfying $gB_-=\gc B_-$. Recall that \cref{lemma:Xggc} gives an explicit procedure to recover $\gc$ from $g$.

In view of \cref{lemma:Xggc} and~\eqref{eq:gcT_dw_in_B_Ym}, we consider factorizations
\begin{equation}\label{eq:YpYoYm_dfn}
  g^T\dw=\Yp\Yo \Ym \quad\text{and}\quad    \gc^T\dw=\Ypp\Yo \Ym
\end{equation}
for $(\Yp,\Yo,\Ym)\in N\times \H\times N_-$ and $\Ypp=\Xggc^{T}\Yp\in N$, where $x$ is given by \cref{lemma:Xggc}. See~\eqref{eq:ex_YpYoYm} for an example.  By definition, we have
\begin{equation}\label{eq:yR=Yp}
  \yR=[g^T\dw]_L^+=\Yp.
\end{equation}

Our first goal is to express the element $\Yo$ in terms of the MR-parameters $\bt\in(\R_{>0})^{\Jo}$. 
\begin{lemma}\label{lemma:Yo_tnn}
The element $\Yo\in \H$ is given by
\begin{equation*}%
  \Yo=\prod_{j\in\Jo} t_j^{-s_{i_m}\cdots s_{i_{j+1}}\ach_{i_j}}.
\end{equation*}
In particular, we have $\Yo\in\Htp$ and for each $i\in I$, 
\begin{equation}\label{eq:dom_Yo}
  \dom(\Yo)=\prod_{j\in \Jo} t_j^{-\<\om_i,s_{i_m}\cdots s_{i_{j+1}}\ach_{i_j}\>}.
\end{equation}
\end{lemma}
\begin{proof}
We have $\Yo=[g^T\dw]_0^\pm=[g_m^T\cdots g_1^T\ds_{i_1}\cdots \ds_{i_m}]_0^\pm$. We prove the following stronger statement by induction on $\ell(w)$: for any $h\in N$, we have
\begin{equation}\label{eq:Y0_proof}
  [g^T h \dw]_0^\pm=\prod_{j\in\Jo} t_j^{-s_{i_m}\cdots s_{i_{j+1}}\ach_{i_j}}.
\end{equation}
 In the base case $w=\id$, both sides of~\eqref{eq:Y0_proof} are equal to the identity element of $\H$. Assume now that the result is known for all $w'\in W$ with $\ell(w')<\ell(w)$. We consider two cases: either $g_1^T=\ds_i^{-1}$ or $g_1^T=y_i(t)$ for some $i\in I$ and $t>0$. In each case, we set $w':=s_iw<w$ and $g':=g_2\cdots g_m$.

In the first case $g_1^T=\ds_i^{-1}$, we have
\begin{equation*}%
  [g^T h \dw]_0^\pm=[(g')^T\ds_i^{-1}h\ds_{i}\dw']_0^\pm.
\end{equation*}
By \cref{conj4}, we can factorize $\ds_i^{-1}h\ds_{i}=h'y_i(t')$ for $h'\in N$ and $t'\in\R$. By \cref{conj2}, $y_i(t')\dw'\in \dw' N_-$, and thus we get $[g^T h \dw]_0^\pm=[(g')^Th'\dw']_0^\pm$. The result follows from the induction hypothesis.

Consider the second case $g_1^T=y_i(t)$. Using \cref{lemma:factor_NN'}, we factorize $h=h_1h_2$ for $(h_1,h_2)\in N'(s_i)\times N(s_i)$. Thus $h_2=x_i(t')$ for some $t'\in\C$. Since $s_iw<w$, by \cref{conj1}, we get $h_2\dw\in \dw N_-$, and thus $[g^T h \dw]_0^\pm=[g^T h_1 \dw]_0^\pm$. Write
\begin{equation*}%
  [g^T h_1 \dw]_0^\pm=[(g')^Ty_i(t)h_1\ds_{i}\dw']_0^\pm=[(g')^T\cdot y_i(t)h_1 y_i(-t) \cdot y_i(t)\ds_{i}\dw']_0^\pm.
\end{equation*}
Applying a collision move~\eqref{eq:collision_y}, we write $y_i(t)\ds_i=x_i(t_+)a_+^{-1}y_i(t_-)$ where $t_+=1/t$ and $a_+=t^{\ach_i}$. By \cref{conj2}, $y_i(t_-)\dw'\in\dw' N_-$. Next, we have $t^{-\ach_i}\dw'=\dw' t^{-s_{i_m}\cdots s_{i_2}\ach_i}$. By \cref{conj3}, we have $y_i(t)h_1 y_i(-t) \in N$. Setting $h':=y_i(t)h_1y_i(-t)x_i(t_+)$, we get 
\begin{equation*}%
  [g^T h_1 \dw]_0^\pm=[(g')^Th' \dw']_0^\pm\cdot t^{-s_{i_m}\cdots s_{i_2}\ach_i}.
\end{equation*}
This finishes the induction step.
\end{proof}

We now use the above computation to show that the twist maps preserve total positivity.
\begin{proof}[Proof of \cref{thm:twist_tnn}]
In view of \cref{thm:chiral_tnn}, it suffices to show that the right pre-twist map $\tpre_v^w:\Rich_v^w\xrasim \LRich_v^w$ sends $\Rtp_v^w$ to $\LRtp_v^w$.  Recall from~\eqref{eq:gcT_dw_in_B_Ym} that the element $\Ym$ belongs to $\dw^{-1}N\dw\cap N_-$. Let $h:=\dw \Ym^T$, and thus $h\in N_-\dw\cap \dw N$. By~\eqref{eq:BgT=BYmdwi}, we have $Bg^T=B h^T$, and thus $gB_-=h B_-$ belongs to $\GBtnn$.
 By \cref{thm:chiral_tnn_op}, we get %
\begin{equation*}%
  B_-h^{-T}=B_- \dw \Ym^{-1} \quad \in \BGtnn.
\end{equation*}
Since $\Yo\in\Htp$ by \cref{lemma:Yo_tnn}, we can act by $\Yo^{-1}$ on the right and get $B_- \dw \Ym^{-1} \Yo^{-1}\in\BGtnn$. By~\eqref{eq:YpYoYm_dfn}, $\Yo\Ym=\Yp^{-1}g^T\dw$, so $\Ym^{-1}\Yo^{-1}=\dw^{-1}g^{-T}\Yp$. We have shown that $B_-g^{-T}\Yp\in \BGtnn$. Since $g\in N\dv$, we have $g^{-T}\in N_-\dv$, so $B_-g^{-T}\Yp=B_- v \Yp\in \BGtnn$. It remains to note that the right pre-twist of $gB_-$ is given by 
\begin{equation*}%
 \tpre_v^w(gB_-):=B_-v[g^T\dw]_L^+=B_-v\Yp\quad \in\BGtnn.\qedhere
\end{equation*}
\end{proof}

\section{Right twist Chamber Ansatz}\label{sec:right-twist-chamber}
We concentrate on proving \cref{thm:R_chamber_ans}. Even though it looks similar to the Marsh--Rietsch Chamber Ansatz (\cref{thm:MR_chamber_ans}), we were unable to formally deduce one statement from the other. Thus, we utilize the proof strategy of~\cite{BZ} who treated the case $v=\id$.

Throughout, let us fix a pair $v\leq w$ in $W$, a reduced expression $\bw$ for $w$, and an MR-parametrized element $g=\gt=g_1g_2\cdots g_m$ with $\bt\in(\Cast)^{\Jo}$. Our goal is to study the element $\yR:=[g^T\dw]_L^+$. Recall from~\eqref{eq:NN'_dfn} that for $u\in W$, we set $N(u):=N\cap \du^{-1}N_-\du$ and $N'(u):=N\cap \du^{-1}N\du$. 

\begin{lemma}\label{lemma:yR_yRpr_N'(u)}
Let $1\leq r\leq m+1$ and 
\begin{equation}\label{eq:BFZ_step1_statement}
\yRpr:=[g_m^T\cdots g_r^T\ds_{i_r}\cdots \ds_{i_m}]_L^+.
\end{equation}
Then $\yR\in \yRpr\cdot N'(s_{i_r}\cdots s_{i_m})$.
\end{lemma}
\begin{proof}
We first show that for all $h\in N$, we have
\begin{equation}\label{eq:BFZ_step_1x}
  [g^T h \dw]_L^+\in \yR \cdot N'(w).
\end{equation}
By \cref{lemma:factor_NN'}, we can factorize $h=h_1h_2$ for $(h_1,h_2)\in N'(w^{-1})\times N(w^{-1})$. Since $h_2\dw\in \dw N_-$, we have $[g^Th\dw]_L^+=[g^Th_1\dw]_L^+$. Next, by~\eqref{eq:YpYoYm_dfn}, we have
\begin{equation*}%
  g^Th_1\dw=g^T\dw\cdot \dw^{-1}h_1\dw=\Yp\Yo\Ym\cdot \dw^{-1}h_1\dw.
\end{equation*}
By~\eqref{eq:gcT_dw_in_B_Ym}, $\Ym\in N_- \cap \dw^{-1}N \dw$. Note that $\dw^{-1}h_1\dw\in N'(w)$. By \cref{lemma:factor_NN'}, the multiplication map gives isomorphisms 
\begin{equation*}%
  (N_- \cap \dw^{-1}N \dw) \times N'(w) \to \dw^{-1}N \dw \to N'(w) \times (N_- \cap \dw^{-1}N \dw).
\end{equation*}
Thus
\begin{equation*}%
  g^Th_1\dw\in \Yp\Yo \cdot N'(w) \cdot  (N_- \cap \dw^{-1}N \dw)\subset \Yp \cdot N'(w)\cdot \H \cdot  (N_- \cap \dw^{-1}N \dw).
\end{equation*}
We therefore find
\begin{equation*}%
  [g^T h \dw]_L^+=[g^Th_1\dw]_L^+\in \Yp\cdot N'(w).
\end{equation*}
It remains to note that by~\eqref{eq:yR=Yp}, $\Yp=\yR$. %

Clearly,~\eqref{eq:BFZ_step_1x} holds with $(w,g,\yR)$ replaced with $(s_{i_r}\cdots s_{i_m},g_r\cdots g_m,\yRpr)$: for all $h\in N$, we have
\begin{equation}\label{eq:BFZ_step_1y}
  [g_m^T\cdots g_r^T h \ds_{i_r}\cdots \ds_{i_m}]_L^+\in \yRpr \cdot N'(s_{i_r}\cdots s_{i_m}).
\end{equation}

Finally, we show that for all $h^{(1)}\in N$ and all $1\leq r\leq m+1$, there exists $h^{(r)}\in N$ such that 
\begin{equation}\label{eq:BFZ_step_1z}
  [g^T h^{(1)} \dw]_L^+=[g_m^T\cdots g_r^T h^{(r)} \ds_r\cdots \ds_m]_L^+.
\end{equation}
We proceed by induction on $r$. The base case $r=1$ is clear. To pass from $r=1$ to $r=2$, we consider two cases: either $g_1^T=\ds_i^{-1}$ or $g_1^T=y_i(t)$ for some $i\in I$ and $t>0$. In each case, we set $w':=s_iw<w$ and $g':=g_2\cdots g_m$. The proof of \cref{lemma:Yo_tnn} shows that in each case, we have $[g^T h^{(1)} \dw]_L^+=[(g')^Th^{(2)}\dw']_L^+$ for some $h^{(2)}\in N$. Iterating this procedure, we prove~\eqref{eq:BFZ_step_1z} for all $1\leq r\leq m+1$.

To finish the proof of the lemma, we observe that $\yR=[g^T h^{(1)} \dw]_L^+$ for $h^{(1)}=1\in N$. By~\eqref{eq:BFZ_step_1z}, $\yR=[g_m^T\cdots g_r^T h^{(r)} \ds_r\cdots \ds_m]_L^+$ for some $h^{(r)}\in N$. Applying~\eqref{eq:BFZ_step_1y}, we get $\yR\in \yRpr \cdot N'(s_{i_r}\cdots s_{i_m})$.
\end{proof}

\begin{corollary}\label{cor:BFZ_step1}
For each $1\leq r\leq m$, we have 
\begin{equation*}%
  \fR_r(\yR)=\fR_r(\yRpr).
\end{equation*}
\end{corollary}
\begin{proof}
Recall from~\eqref{eq:fR_dfn} that $\fR_r=\DOMir^{v\pu{r-1}}_{w\pu{r-1}}=\DOMir^{\sv_{i_m}\cdots \sv_{i_r}}_{s_{i_m}\cdots s_{i_r}}$. 
 Applying \cref{lemma:Lec_NN'} to the pair $(\sv_{i_r}\cdots\sv_{i_m}\leq s_{i_r}\cdots s_{i_m})$, we see that the function $\fR_r\in\C[N]$ is invariant under right multiplication by $N'(s_{i_r}\cdots s_{i_m})$. The result follows from \cref{lemma:yR_yRpr_N'(u)}.
\end{proof}

\begin{lemma}\label{lemma:g_minors}
For all $i\in I$, we have
\begin{equation*}%
  \DOM^{}_{w^{-1}}(g)=\prod_{r\in\Jo} t_r^{ \<\om_i,s_{i_1}\cdots s_{i_{r-1}}\ach_{i_r} \>  }
\end{equation*}
\end{lemma}
\begin{proof}
We have
\begin{equation*}%
    \DOM^{}_{w^{-1}}(g)=  \dom(g\dw^{-1})=\dom([g\dw^{-1}]_0^{\mp}).
\end{equation*}
The computation of $\dom([g\dw^{-1}]_0^{\mp})$ is entirely analogous to the computation of $\dom(y_0)=\dom([g^T\dw]_0^{\pm})$ which was completed in \cref{lemma:Yo_tnn}.
\end{proof}

\begin{lemma}\label{lemma:DOM_DOM=1}
For all $i\in I$, we have
\begin{equation*}%
  \DOM^{v^{-1}}_{w^{-1}}(\yR)\cdot \DOM^{}_{w^{-1}}(g)=1.
\end{equation*}
\end{lemma}
\begin{proof}
Let $\hat y:=\dv \yR$. By \cref{prop:pre_pre_twist}, we have $\hat y\in \capBWB(\dv N)$. By~\eqref{eq:Gauss_dwi_g}, we have $\dw \hat y^T\in \Gomp$. Let $(\Xm,\Xo ,\Xp)\in N_-\times \H\times N$ be given by 
\begin{equation}\label{eq:dwyTdv}
  \dw \yR^{\,T} \dv^{-1}=\dw \hat y^T=\Xm\Xo \Xp.%
\end{equation}
Our first goal is to show that $\Xp=g\dv^{-1}$. Since $\Xp\in N$ and $g\in N\dv$, it suffices to show that $B_-\Xp=B_-g\dv^{-1}$. Recall from~\eqref{eq:YpYoYm_dfn} that we have a factorization $g^T\dw=\Yp\Yo \Ym$, and as explained in~\eqref{eq:gcT_dw_in_B_Ym}, we have $\Ym\in \dw^{-1}N\dw\cap N_-$. In particular, $\Ym^T\in\dw^{-1}N_-\dw$, and thus $\dw \Ym^T\Yo^T\dw^{-1}\in B_-$. By~\eqref{eq:yR=Yp}, we have $\yR=\Yp$. Applying~\eqref{eq:dwyTdv}, we find
\begin{equation*}%
  B_-\Xp=B_-\Xm\Xo\Xp=B_-\dw\yR^{\,T}\dv^{-1}=B_- \dw \Ym^T\Yo^T\dw^{-1}\cdot \dw\yR^{\,T}\dv^{-1}=B_- \dw \cdot \dw^{-1} g\dv^{-1}=B_- g\dv^{-1}.
\end{equation*}
Thus indeed $\Xp=g\dv^{-1}$. 

By \cref{lemma:factor_NN'}, $\dw \yR^{\,T} \dw^{-1}\in N_- N$. By~\eqref{eq:dom_Nmp}--\eqref{eq:dom_H}, we have
\begin{align*}%
  1&=\dom(\dw \yR^{\,T} \dw^{-1})=\dom(\dw (\dw^{-1}\Xm\Xo \Xp \dv) \dw^{-1})=\dom(\Xo \Xp\dv\dw^{-1})=\dom(\Xo )\dom(\Xp\dv\dw^{-1})\\
&=\dom(\Xo )\dom(g\dw^{-1})=\dom(\Xo )\DOM^{}_{w^{-1}}(g).
\end{align*}
It remains to show that $\dom(\Xo )=\DOM^{v^{-1}}_{w^{-1}}(\yR)$. From~\eqref{eq:dwyTdv}, we find $\yR=\dv^{-1}\Xp^T\Xo \Xm^T\dw$, and thus
\begin{equation*}%
  \DOM^{v^{-1}}_{w^{-1}}(\yR)=\dom(\dv(\dv^{-1}\Xp^T\Xo \Xm^T\dw)\dw^{-1})=\dom(\Xp^T\Xo \Xm^T)=\dom(\Xo ),
\end{equation*}
since $\Xp^T\in N_-$ and $\Xm^T\in N$.%
\end{proof}

\begin{proof}[Proof of \cref{thm:R_chamber_ans}]
We first show~\eqref{eq:R_chamber_ans_inv}. By \cref{cor:BFZ_step1}, we have $\fR_j(\yR)=\fR_j(y\pj)$. Let $g\pj:=g_j\cdots g_m$ and let $\yRpj$ be given by~\eqref{eq:BFZ_step1_statement}. By \cref{lemma:DOM_DOM=1}, we have
\begin{equation*}%
  \fR_j(\yRpj)\cdot \Delta^{\om_{i_j}}_{s_{i_m}\cdots s_{i_{j}} \om_{i_j}}(g\pj)=1.
\end{equation*}
By \cref{lemma:g_minors},  
\begin{equation*}%
  \Delta^{\om_{i_j}}_{s_{i_m}\cdots s_{i_{j}} \om_{i_j}}(g\pj)=\prod_{r\in\Jo:\, r\geq j } t_r^{\<\om_{i_j},s_{i_j}s_{i_{j+1}}\cdots s_{i_{r-1}} \ach_{i_r}\>}.
\end{equation*}
This shows~\eqref{eq:R_chamber_ans_inv}.

Let us now deduce~\eqref{eq:R_chamber_ans} from~\eqref{eq:R_chamber_ans_inv}. Each side of~\eqref{eq:R_chamber_ans} is a monomial in $(t_k)_{k\geq r}$. Let $k\in \Jo$ be such that $k\geq r$. We will show that $t_k$ appears on both sides of~\eqref{eq:R_chamber_ans} with the same exponent. Assume first that $k>r$. By~\eqref{eq:R_chamber_ans_inv}, the exponent of $t_k$ in the numerator of the right-hand side of~\eqref{eq:R_chamber_ans} is given by
\begin{equation*}%
  \<\sum_{j\neq i_r} a_{j,i_r}\om_j,s_{i_{r+1}}\cdots s_{i_{k-1}}\ach_{i_k}\>.
\end{equation*}
On the other hand, the exponent of $t_k$ in the denominator is
\begin{equation*}%
  \<-\om_{i_r},s_{i_r}\cdots s_{i_{k-1}}\ach_{i_k}\>+\<-\om_{i_r},s_{i_{r+1}}\cdots s_{i_{k-1}}\ach_{i_k}\>.
\end{equation*}
Since $s_{i_r}\om_{i_r}=\om_{i_r}-\al_{i_r}$, we find
\begin{equation*}%
  \<-\om_{i_r},s_{i_r}\cdots s_{i_{k-1}}\ach_{i_k}\>=\<\al_{i_r}-\om_{i_r},s_{i_{r+1}}\cdots s_{i_{k-1}}\ach_{i_k}\>.
\end{equation*}
Therefore the total exponent of $t_k$ in the right-hand side of~\eqref{eq:R_chamber_ans} is
\begin{equation*}%
  \<-\al_{i_r}+2\om_{i_r}+\sum_{j\neq i_r} a_{j,i_r}\om_j,s_{i_{r+1}}\cdots s_{i_{k-1}}\ach_{i_k}\>=  \<-\al_{i_r}+\sum_{j\in I} a_{j,i_r}\om_j,s_{i_{r+1}}\cdots s_{i_{k-1}}\ach_{i_k}\>=0,
\end{equation*}
since $\al_{i_r}=\sum_{j\in I} a_{j,i_r}\om_j$. We have matched the exponent of $t_k$ on both sides of~\eqref{eq:R_chamber_ans} for all $k>r$. For $k=r$, by~\eqref{eq:R_chamber_ans_inv}, $t_r$ appears in $\fR_j(\yR)$ only for $j\leq r$. Thus, out of all terms in the right-hand side of~\eqref{eq:R_chamber_ans}, it only contributes to $\fR_r(\yR)=\Delta^{\Vi{r-1}\om_{i_r}}_{\Wi{r-1}\om_{i_r}}(\yR)$ in the denominator. It is clear from~\eqref{eq:R_chamber_ans_inv} that the exponent of $t_r$ in $\fR_r(\yR)$ is equal to $-1$.
\end{proof}

Comparing \cref{thm:R_chamber_ans,thm:MR_chamber_ans} with \cref{lemma:Yo_tnn}, we obtain the following relation between the left and the right twist Chamber Ansatz formulas. Set
\begin{equation}\label{eq:z_wyw}
  z:=(\dw \Ym \dw^{-1})^T.
\end{equation}
 Thus $z\in N_-\cap \dw N\dw^{-1}$. By~\eqref{eq:BgT=BYmdwi}, we have $Bg^T=B\dw^{-1}z^T$, and thus $gB_-=z\dw B_-$ belongs to $\GBtnn$. We conclude that the element $z$ satisfies the assumptions of \cref{thm:MR_chamber_ans}. \Cref{thm:R_chamber_ans} is stated in terms of the element $\yR:=[g^T\dw]_L^+$, which by~\eqref{eq:yR=Yp} equals $\Yp$. 

\begin{corollary}\label{cor:R_vs_L_ans}
For all $i\in I$ and $r=1,\dots,m+1$, we have
\begin{equation}\label{eq:minors_YpYo}
  \Delta^{\sv_{i_m}\cdots \sv_{i_r} \om_i}_{s_{i_m}\cdots s_{i_r}\om_i}(\Yp\Yo)=
\Delta^{s_1\cdots s_{i_{r-1}} w_0\om_{i^\ast}}_{\sv_1\cdots\sv_{i_{r-1}} w_0\om_{i^\ast}}(z^T).
\end{equation}
\end{corollary}
\begin{remark}
\Cref{cor:R_vs_L_ans} demonstrates the primary difficulty with relating the two Chamber Ansatz formulas: \cref{thm:R_chamber_ans} is given in terms of $\yR=\Yp$, while \cref{thm:L_chamber_ans} is given in terms of $\yL$ which is closely related to $\Ym$, as explained in the proof of \cref{thm:L_chamber_ans}. It is unclear to us how to show~\eqref{eq:minors_YpYo} directly without reproving each Chamber Ansatz from scratch.
\end{remark}
\begin{remark}
Let $G=\SL_n(\C)$. Let $C$ be a chamber of the wiring diagram of $\bw$ located vertically at height $i$ and horizontally between the $(r-1)$-th and the $r$-th crossings. Then the left- and the right-hand sides of~\eqref{eq:minors_YpYo} are equal to $\chmnrR_C(\Yp\Yo)$ and $\chmnrL_C(z^T)$, respectively.
\end{remark}

\begin{example}\label{ex:R_vs_L}
We continue the example from \cref{sec:A_example}. Recall that the matrices $\Yp,\Yo,\Ym$ defined in~\eqref{eq:YpYoYm_dfn} have been computed in the right-hand side of~\eqref{eq:ex_YpYoYm}. Thus
\begin{equation*}%
 \Yp= \smat{
1 & \frac{1}{t_{6}} & \frac{t_{3}}{t_{1}} & 0 \\
0 & 1 & \frac{t_{3} t_{6}}{t_{1}} & -\frac{1}{t_{4}} \\
0 & 0 & 1 & \frac{1}{t_{3} t_{4}} \\
0 & 0 & 0 & 1
} \quad\text{and}\quad z^T=\dw \Ym \dw^{-1}=\smat{
1 & \frac{t_{1}}{t_{3}} & \frac{1}{t_{3}} & \frac{1}{t_{3} t_{4}} \\
0 & 1 & \frac{1}{t_{1}} & \frac{1}{t_{1} t_{4}} \\
0 & 0 & 1 & \frac{t_{1} + t_{6}}{t_{4} t_{6}} \\
0 & 0 & 0 & 1
}.
\end{equation*}
The minors of these matrices appearing in \cref{cor:R_vs_L_ans} are computed in \cref{fig:R_vs_L}.
\end{example}

\begin{figure}
\begin{tabular}{c}%
\includegraphics[width=0.98\textwidth]{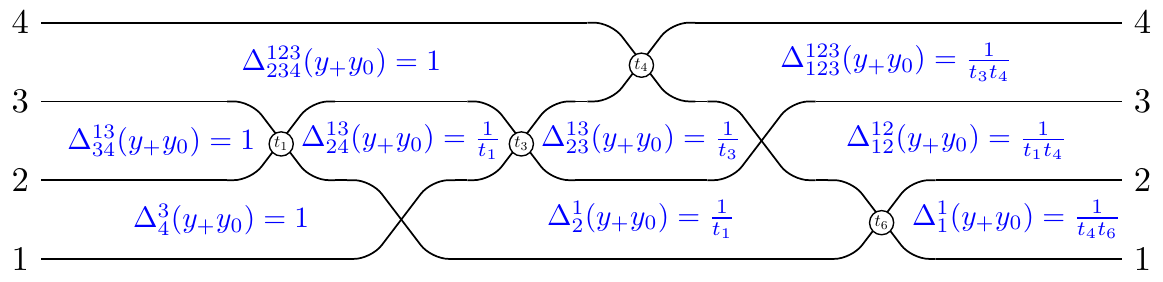}%
\\
\\
\includegraphics[width=0.98\textwidth]{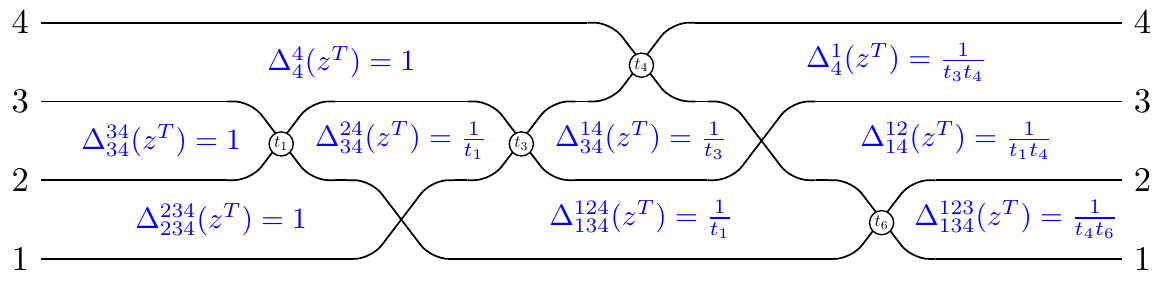}%
\end{tabular}
  \caption{\label{fig:R_vs_L}Relating the left and the right twist Chamber Ansatz formulas; cf. \cref{cor:R_vs_L_ans} and \cref{ex:R_vs_L}.}
\end{figure}

\section{Muller--Speyer twist}\label{sec:muller-speyer-twist}
We relate our twist maps to the twist maps of~\cite{MuSp}.

\subsection{Background on positroid varieties}
We continue using the notation from \cref{sec:MuSp_main_res}. Recall that for $g\in G$, $\pi_J(gB_-)\in\Gr(n-k,n)$ is the linear span of the rightmost $n-k$ columns of the matrix $g$. We denote by $|g]$ the corresponding $n\times (n-k)$ submatrix of $g$ with column set $[k+1,n]:=\{k+1,k+2,\dots,n\}$. For an $(n-k)$-element set $F\subset[n]$ and $g\in G$, we let $\Delta_F(g)$ be the minor of $g$ with row set $F$ and column set $[k+1,n]$. The functions $\Delta_F(g)$ give projective coordinates on $\Gr(n-k,n)$ called \emph{Pl\"ucker coordinates}.

Recall that for $(v,w)\in\QJ$, we have the open positroid variety $\Pio_v^w=\pi_J(\Rich_v^w)$. The \emph{positroid cell} $\Povtp_v^w$ is the image of $\Rtp_v^w$ under $\pi_J$. Alternatively, it can be described as the subset of $\Pio_v^w$ where all nonzero Pl\"ucker coordinates have the same sign.

\subsection{The Muller--Speyer twist map}\label{sec:MuSp_dfn}
Fix $(v,w)\in\QJ$. We define the \emph{right twist} $\TMSR_v^w:\Pio_v^w\xrasim \Pio_v^w$. Let $g\in G$ be such that $g\PJ\in\Pio_v^w$, and let $M:=|g]$ be the corresponding $n\times (n-k)$ matrix. Thus $g\PJ$ is the column span of $M$. We denote by $M_1,M_2,\dots,M_n\in \C^{n-k}$ the row vectors of $M$. We extend this labeling periodically for all $i\in \Z$ by the condition that $M_{i+n}=M_i$ for all $i\in\Z$. Let $\<\cdot,\cdot\>$ denote the standard inner product on $\C^{n-k}$. For $i\in \Z$, let 
\begin{align*}%
  \IR_i&:=\{j\geq i\mid M_j\text{ is not in the span of $M_i,M_{i+1},\dots,M_{j-1}$}\};\\
  \IL_i&:=\{j\leq i\mid M_j\text{ is not in the span of $M_i,M_{i-1},\dots,M_{j+1}$}\}.\\
\end{align*}
By definition, the span of the empty set is $\{0\}$, and thus when $M_i=0$, we have $i\notin \IR_i$ and $i\notin \IL_i$. We always have $|\IR_i|=|\IL_i|=n-k$ and each of the sets $\{M_j\mid j\in\IR_i\}$ and $\{M_j\mid j\in \IL_i\}$ forms a basis of $\C^{n-k}$.

For $i\in \Z$, let $\vec M_i\in\C^{n-k}$ be the unique vector given by 
\begin{equation*}%
  \<\vec M_i,M_j\>=
  \begin{cases}
    1, &\text{if $i=j$},\\
    0, &\text{otherwise},
  \end{cases} \quad\text{for all $j\in\IR_i$}.
\end{equation*}
Similarly, let $\cev M_i\in\C^{n-k}$ be the unique vector given by 
\begin{equation*}%
  \<\cev M_i,M_j\>=
  \begin{cases}
    1, &\text{if $i=j$},\\
    0, &\text{otherwise},
  \end{cases} \quad\text{for all $j\in\IL_i$}.
\end{equation*}
Let $\vec M$ be the matrix with rows $\vec M_1,\vec M_2,\dots,\vec M_n$, and let $\cev M$ be the matrix with rows $\cev M_1,\cev M_2,\dots,\cev M_n$. By definition, $\TMSR_v^w(g\PJ)$ is the column span of $\vec M$ and $\TMSL_v^w(g\PJ)$ is the column span of $\cev M$. See \cref{ex:MuSp} below.

\subsection{Relating the twist maps}
The following result, giving the relationship between our twist maps $\twist_v^w,\twistL_v^w$ and the twist maps $\TMSR_v^w,\TMSL_v^w$ of~\cite{MuSp}, is a reformulation of \cref{thm:MuSp_twist}.
\begin{theorem}
  For all $(v,w)\in\QJ$, we have the following commutative diagram, where all maps are isomorphisms:
\begin{equation}\label{eq:MuSp_cd}
\begin{tikzcd}%
\Rich_v^w \arrow[d,"\pi_J"] & \Rich_v^w \arrow[r,rightarrow,"\twist_v^w"] \arrow[d,"\pi_J"] \arrow[l,rightarrow,"\twistL_v^w"'] &\Rich_v^w \arrow[d,"\pi_J"]\\
\Pio_v^w & \Pio_v^w \arrow[r,rightarrow,"\TMSR_v^w"'] \arrow[l,rightarrow,"\TMSL_v^w"] & \Pio_v^w.
\end{tikzcd}  
\end{equation}
\end{theorem}
\begin{proof}
We focus on the commutativity of the right square in~\eqref{eq:MuSp_cd}. First, we claim that the following diagram is commutative, where all maps are homeomorphisms:
\tikzset{close/.style={outer sep=-2pt}} %
\begin{equation}\label{eq:MuSp_inside_cd}
\begin{tikzcd}%
\Rtp_v^w \arrow[rrr,"\twist_v^w"] \arrow[dd,"\pi_J"] &&&
\Rtp_v^w \arrow[dd,"\pi_J"] \arrow[dl,"(\fR_r)_{r\in\Jo}"',close,pos=0.5]\\
 & (\R_{>0})^{\Jo} \arrow[r,"\text{\eqref{eq:R_chamber_ans_inv}}", bend left=10] 
\arrow[ul,"\text{\eqref{eq:MR}}"',close] 
\arrow[dl,"\Meas",close]
& 
(\R_{>0})^{\Jo} \arrow[l,"\text{\eqref{eq:R_chamber_ans}}", bend left=10] & \\
\Povtp_v^w \arrow[rrr,"\TMSR_v^w"']  &&& 
\Povtp_v^w \arrow[ul,close,"(\Delta_{F_r})_{r\in\Jo}",pos=0.5].
\end{tikzcd}  
\end{equation}
Here $\Meas: (\R_{>0})^{\Jo}\xrasim \Povtp_v^w$ is the \emph{boundary measurement map} of~\cite{Pos}, which puts the positive parameters $(t_r)_{r\in \Jo}$ on the edges of a \emph{Le-diagram plabic graph}; see part (c) in the proof of~\cite[Proposition~4.3]{GL_cluster}. We have a collection $\{F_0\}\sqcup\{F_r\mid r\in\Jo\}$ of $(n-k)$-element sets labeling the faces of the Le-diagram as in~\cite[Section~1.3]{GL_cluster}. The corresponding Pl\"ucker coordinates are functions on $\Pio_v^w$ defined up to common rescaling, and we always rescale them so that $\Delta_{F_0}=1$. Thus we obtain a map $(\Delta_{F_r})_{r\in\Jo}:\Povtp_v^w\to (\R_{>0})^{\Jo}$, which appears as the bottom right diagonal arrow in~\eqref{eq:MuSp_inside_cd}.

The commutativity of the top trapezoid in~\eqref{eq:MuSp_inside_cd} follows from the right twist Chamber Ansatz (\cref{thm:R_chamber_ans}). The bottom trapezoid is commutative by the results of~\cite{MuSp}. The left triangle is commutative by the results of~\cite{TW,Karpman2}; see the proof of~\cite[Proposition~4.3]{GL_cluster}. Finally, the right triangle is commutative by~\cite[Lemma~3.6]{GL_cluster}. We have shown that the right square in~\eqref{eq:MuSp_cd} is commutative when all maps are restricted to $\Rtp_v^w$ and $\Povtp_v^w$. Since these maps are regular and the subsets $\Rtp_v^w\subset \Rich_v^w$ and $\Povtp_v^w\subset \Pio_v^w$ are Zariski dense, it follows that the right square in~\eqref{eq:MuSp_cd} is commutative. Since $\twistL_v^w$ and $\TMSL_v^w$ are inverses of $\twist_v^w$ and $\TMSR_v^w$, respectively, we see that the left square in~\eqref{eq:MuSp_cd} is commutative as well.
\end{proof}

\begin{example}\label{ex:MuSp}
Let $k=2$, $n=5$, $v=s_4s_2$, and $\bw=(2,1,4,3,2)$. (Thus $w=s_2s_1s_4s_3s_2$ is $k$-Grassmannian since $ws_j>w$ for all $j\neq k$.) We have
\begin{equation*}%
  g=\smat{
1 & 0 & t_{2} & 0 & 0 \\
0 & -t_{1} & 1 & t_{1} t_{4} & 0 \\
0 & -1 & 0 & t_{4} & 0 \\
0 & 0 & 0 & 0 & 1 \\
0 & 0 & 0 & -1 & 0
},\quad \twist_v^w(gB_-)=\smat{
1 & \frac{1}{t_{1} t_{2}} & \frac{1}{t_{2}} & 0 & 0 \\
0 & 0 & 1 & \frac{1}{t_{1} t_{4}} & 0 \\
0 & -1 & 0 & \frac{1}{t_{4}} & 0 \\
0 & 0 & 0 & 0 & 1 \\
0 & 0 & 0 & -1 & 0
}\cdot B_-.
\end{equation*}
Applying the map $x\mapsto |x]$, we get
\begin{equation*}%
  \pi_J(gB_-)=\colspan \smat{
t_{2} & 0 & 0 \\
1 & t_{1} t_{4} & 0 \\
0 & t_{4} & 0 \\
0 & 0 & 1 \\
0 & -1 & 0
} \quad\text{and}\quad \pi_J\circ\twist_v^w(gB_-)=\colspan\smat{
\frac{1}{t_{2}} & 0 & 0 \\
1 & \frac{1}{t_{1} t_{4}} & 0 \\
0 & \frac{1}{t_{4}} & 0 \\
0 & 0 & 1 \\
0 & -1 & 0
}.
\end{equation*}
On the other hand, applying the definition in \cref{sec:MuSp_dfn} to the matrix $M=|g]$, we find
\begin{equation*}%
  \TMSR_v^w\circ \pi_J(gB_-)=\colspan \vec M, \quad\text{where}\quad \vec M=\smat{
\frac{1}{t_{2}} & -\frac{1}{t_{1} t_{2} t_{4}} & 0 \\
1 & 0 & 0 \\
0 & \frac{1}{t_{4}} & 0 \\
0 & 0 & 1 \\
0 & -1 & 0
}
\end{equation*}
Applying column operations, we see that $\pi_J\circ\twist_v^w(gB_-)=\TMSR_v^w\circ \pi_J(gB_-)$ inside $\Gr(n-k,n)$, confirming the commutativity of the right square in~\eqref{eq:MuSp_cd}.
\end{example}

\section{Subtraction-free expressions}\label{sec:SF}
In the previous sections, we have shown that several maps preserve total positivity. In this section, we show the stronger statement that these maps are given by \emph{subtraction-free rational functions} when written in terms of the MR-parameters. In order to state our results precisely, we review the results of~\cite{BaoHe}.

Let $\kbf$ denote a field.  We let $G(\kbf), B(\kbf), B_-(\kbf), (G/B_-)(\kbf), \Rich_v^w(\kbf),$ and so on denote the respective groups or varieties defined over $\kbf$.  The chiral maps $\chiv, \chivi$, the opposite chiral maps $\chiw, \chiwi$, the pre-twists $\tpre_v^w, \Ltpre_v^w$, the twists $\twist_v^w, \twistL_v^w$, and the reversal maps $\revR,\revL$ of \cref{sec:reversal_maps} are all defined over $\kbf$.

Let $\bt=(t_1,t_2,\dots,t_a)$ be a collection of variables. A rational function $\phi\in\Q(\bt)$ is called \emph{$\Z$-subtraction free} if it can be written as a ratio $P(\bt)/Q(\bt)$ where $P,Q$ are polynomials in $\bt$ with nonnegative integer coefficients. A rational map $f: \k^a \dashrightarrow \k^b$ is called \emph{$\Z$-subtraction free} if it can be written as $f(\bx)=(\phi_1(\bx),\phi_2(\bx),\dots,\phi_b(\bx))$ where $\phi_1,\dots,\phi_b$ are $\Z$-subtraction free rational functions. (Such maps are called \emph{admissible} in~\cite{Lus_II,BaoHe}.)

For $(\bv,\bw) \in\Red(v,w)$, let $\gbf_{\bv,\bw}: (\k^*)^{\ell(w)-\ell(v)} \to \Rich_v^w(\kbf)$ denote the MR-parametrization.  It is known (see~\cite[Proposition~6.1]{BaoHe}) that for $(\bv,\bw),(\bv',\bw') \in\Red(v,w)$ the rational map $\gbf_{\bv',\bw'}^{-1} \circ \gbf_{\bv,\bw}: (\k^*)^{\ell(w)-\ell(v)}  \dashrightarrow (\k^*)^{\ell(w)-\ell(v)}$ is $\Z$-subtraction free, and this rational function does not depend on $\k$.  Similarly, we obtain MR-parametrizations of $\LRich_v^w(\kbf)$ by considering \emph{leftmost} distinguished subexpressions for $v$ inside $\bw$, i.e., by applying the map $gB_-\mapsto B_-g^\iota$ of~\eqref{eq:iota_TNN}.

We say that an isomorphism $\phi: \Rich_v^w(\kbf) \xrasim  \Rich_v^w(\kbf)$ is \emph{$\Z$-subtraction free} if the composition
\begin{equation}\label{eq:admissible_dfn}
  \gbf_{\bv,\bw}^{-1} \circ \phi \circ \gbf_{\bv,\bw}:  (\k^*)^{\ell(w)-\ell(v)}  \dashrightarrow (\k^*)^{\ell(w)-\ell(v)}
\end{equation}
is $\Z$-subtraction free.  This notion does not depend on the choice of $(\bv,\bw)$.  Similarly, we define the notion of \emph{$\Z$-subtraction free} for isomorphisms $\phi: \Rich_v^w(\kbf) \to  \LRich_v^w(\kbf)$, and so on.

\begin{theorem}\label{thm:Zsf}
The chiral maps $\chiv, \chivi$, the opposite chiral maps $\chiw, \chiwi$, the pre-twists $\tpre_v^w, \Ltpre_v^w$, the twists $\twist_v^w, \twistL_v^w$, and the reversal maps $\revR,\revL$ are $\Z$-subtraction free. The corresponding $\Z$-subtraction free rational functions in $\Q(\bt)$ appearing in~\eqref{eq:admissible_dfn} are independent of $\k$.
\end{theorem}

Now let $K\subset \k$ be a semifield.  For $v\leq w$ in $W$, we let
\begin{equation*}%
  \Rsf_v^w:=\{\gt B_-(\kbf) \mid \bt\in K^{\Jo}\}\quad \subset (G/B_-)(\kbf).
\end{equation*}
As we mentioned above,  the subset $\Rsf_v^w\subset (G/B_-)(\kbf)$ does not depend on the choice of $(\bv,\bw)\in\Red(v,w)$ by~\cite[Proposition~6.1]{BaoHe}; see also the last sentence of~\cite[Section~3]{BaoHe}. Similarly, we define $\LRsf_v^w\subset (\BG)(\kbf)$ as the image of $\Rsf_v^w$ under the map $gB_-\mapsto B_-g^\iota$. 
We let
\begin{equation*}%
  \GBsf:=\bigsqcup_{v\leq w} \Rsf_v^w \quad\text{and}\quad \BGsf:=\bigsqcup_{v\leq w} \LRsf_v^w.
\end{equation*}

Theorem~\ref{thm:Zsf} has the following immediate corollary.
\begin{corollary}
The chiral maps $\chiv, \chivi$, the opposite chiral maps $\chiw, \chiwi$, the pre-twists $\tpre_v^w, \Ltpre_v^w$, the twists $\twist_v^w, \twistL_v^w$, and the reversal maps $\revR,\revL$ are bijections between the respective semifield points.
\end{corollary}

\begin{remark}
Using the $\Z$-subtraction free rational maps in Theorem~\ref{thm:Zsf}, we may extend the definitions of these maps to flag manifolds and Richardson varieties defined over arbitrary semifields; see \cite[Section 4]{BaoHe}.
\end{remark}

\def\dG{\dot G}
Our proof of \cref{thm:Zsf} will rely on the results of~\cite{BaoHe}. 
Let $\dG$ be a simply-connected, semisimple, complex algebraic group with simply-laced root datum that folds to $G$.  Thus $\dG$ is equipped with an automorphism $\sigma: \dG \to \dG$ so that $\dG^\sigma \cong G$.  We fix a dominant regular integral weight $\la$ for $\dG$. Let $\laVd$ denote the integrable highest weight module of the quantum group with simply laced root datum folding to $G$, defined over $\Acal:=\Z[q,q^{-1}]$.  The $\Acal$-module $\laVd$ is free over $\Acal$ with canonical basis $\Bd$.  Let 
\begin{equation*}%
  \Vd(\k):=\laVd\otimes_{\Acal} \k,
\end{equation*}
where the map $\Acal\to\k$ is given by $q\mapsto 1$. The $\k$-vector space $\Vd(\k)$ is equipped with the image of the canonical basis, which we also denote by $\Bd$. We consider the projectivization $\P(\Vd(\k))$ of $\Vd(\k)$. For $\xi=\sum_{b\in\Bd} \xi_b b$, we denote by $[\xi]$ the corresponding element of $\P(\Vd(\k))$, for which we will be interested in the ratios $\xi_b/\xi_{b'}$.  See~\cite[Section~2G]{BaoHe} for background and further details.

The flag variety $G/B_-(\k)$ can be identified with a subvariety of $\P(\Vd(\k))$; see \cite[Equation~(2-2)]{BaoHe}.  Thus, $\Rich_v^w(\k)$ can be identified with a locally closed subvariety 
\begin{equation*}%
  \PRds\subset\P(\Vd(\k)).
\end{equation*}
By  \cite[Theorem 3.8]{BaoHe}, $\Rsf_v^w$ is identified with the subset
\begin{equation*}%
  \{[\xi]\in\PRds\mid \xi=\sum_{b\in\Bd} \xi_b b,\text{ where $\xi_b/\xi_{b'}\in K$ for all $b,b'\in\Bd$ such that $\xi_b,\xi_{b'}\neq0$}\}.
\end{equation*}
We will similarly identify $\LRich_v^w(\k)\supset \LRsf_v^w$ with $\LPRds\supset \LPRdsf$ inside $\P(\Vd(\k))$. 

\begin{proof}[Proof of \cref{thm:Zsf}]
Let $\phi:\Rich_v^w(\k)\xrasim\Rich_v^w(\k)$ be an isomorphism, and suppose that $[\xi']=\phi([\xi])$. Then the ratios $\xi'_{b'_1}/\xi'_{b'_2}$ can be expressed as rational functions of $\xi_{b_1}/\xi_{b_2}$, where $b_1,b_2,b'_1,b'_2\in\Bd$. For any of the maps we consider, these rational functions are represented by ratios of integer polynomials independent of $\k$. It follows that the rational maps considered in \cref{thm:Zsf} are independent of~$\k$.

\proofsection{Chiral map.} 

Our proof of \cref{thm:chiral_tnn} gives an isomorphism
\begin{equation*}%
\chivi:  \Rsf_v^{w_0}\xrasim \LRsf_v^{w_0}.
\end{equation*}
(The semifield analog of \cref{lemma:y_i(t)gB_-_in_Rtp} is~\cite[Theorem~1.1(3)]{BaoHe}.) However, the remainder of the proof used continuity arguments. We instead proceed as follows. Choose a reduced expression $\bi=(i_1,\dots,i_l)$ for $w_0w^{-1}$ and for $\bt=(t_1,\dots,t_l)$, let $\bx_{\bi}(\bt):=x_{i_1}(t_1)\cdots x_{i_l}(t_l)$. Here $l=\ell(w_0)-\ell(w)$.  We have an isomorphism 
\begin{equation*}%
  K^{\ell(w_0)-\ell(w)}\times \Rsf_v^w\xrasim \Rsf_v^{w_0},\quad (\bt,gB_-)\mapsto \bx_{\bi}(\bt)gB_-.
\end{equation*}
Pick an MR-parametrization $K^{\ell(w)-\ell(v)}\xrasim \Rsf_v^w$ sending $\bt'\mapsto \gbf_{\bv,\bw}(\bt')$. We thus obtain a parametrization 
\begin{equation*}%
  f: K^{\ell(w_0)-\ell(w)}\times K^{\ell(w)-\ell(v)}\xrasim \Rsf_v^{w_0},\quad (\bt,\bt')\mapsto \bx_{\bi}(\bt) \gbf_{\bv,\bw}(\bt')  B_-.
\end{equation*}
Now, $\chivi(f(\bt,\bt'))\in\LRsf_v^{w_0}$, so all ratios $\xi_b/\xi_{b'}$ (with $\xi_b,\xi_{b'}\neq0$) of the coordinates of $[\xi]=\chivi(f(\bt,\bt'))$ belong to $K$ and must be given by $\Z$-subtraction free rational functions in $\Q(\bt,\bt')$.  We now argue that these functions are $\Z$-subtraction free at $\bt=0$.

Let $K^!:=K\sqcup\{0\}$. Observe that for $(\bt,\bt')\in (K^!)^{\ell(w_0)-\ell(w)}\times K^{\ell(w)-\ell(v)}$, we have $f(\bt,\bt')\in ((BvB_-)/B_-)(\k)$, and thus $\chivi(f(\bt,\bt'))\in (B_-\bs (B_-vB))(\k)$ is well defined. In particular, after scaling all coordinates, we may assume that for all $b\in\Bd$, we have $\xi_b\in\Ztnn[\bt,\bt']$, and that moreover, not all $\xi_b$ are divisible by $t_1$.  Thus $\chivi(f(0,t_2,\ldots,t_l,\bt'))$ is represented by $[\xi']$ where $\xi'_b = \xi_b|_{t_1=0} \in \Ztnn[t_2,\ldots,t_l,\bt']$.  Repeating this argument for $t_2,t_3,\ldots,t_l$, we deduce that 
\begin{equation*}%
  \chivi(\gbf_{\bv,\bw}(\bt'))=[\xi''], \quad\text{where}\quad \xi''_b \in \Ztnn[\bt'] \quad \text{for all $b\in \Bd$.}
\end{equation*}
We deduce that $\chivi$ restricts to a bijection $\chivi:\Rsf_v^w\xrasim \LRsf_v^w$ for all $v\leq w$. By the Marsh--Rietsch Chamber Ansatz (cf. \cref{thm:MR_chamber_ans} and~\cite[Remark~2.12]{BaoHe}), we deduce that the map $\chivi:\Rsf_v^w\xrasim \LRsf_v^w$ is $\Z$-subtraction free. The proof for the map $\chiv$ is similar.
 
\proofsection{Reversal maps}
 For the map $\revR$ sending $gB_-\mapsto g^\th\dw_0B_-$, this is~\cite[Theorem~6.4]{BaoHe}. For the map $\revL$, we note that $B_- g \mapsto g^\iota B_-$ induces, by definition, an isomorphism $(\BG)(K) \to (G/B_-)(K)$.  The claim then follows in the same way as in the proof of Proposition~\ref{prop:rev2}.

\proofsection{Opposite chiral map} Replacing $\Rtp_v^w$ by $\Rsf_v^w$, etc. in the proof of \cref{thm:chiral_tnn_op}, we deduce that 
\begin{equation*}%
  \chiw:\LRsf_v^w\xrasim \Rsf_v^w
\end{equation*}
is a bijection. The $\Z$-subtraction freeness follows as before.

\proofsection{Pre-twist} The computation in the proof of \cref{thm:twist_tnn} shows that $\tpre_v^w$ restricts to a bijection $\LRsf_v^w\xrasim \Rsf_v^w$. The claim follows as before.

\proofsection{Twist} Since the twist map is the composition of the pre-twist and the chiral map, this follows immediately.
\end{proof}

\section{Cluster structures on Richardson varieties}\label{sec:cluster}
The goal of this section is to use our results to relate the two conjectural cluster structures on $\Rich_v^w$ described in~\cite{Lec,Ing}. 

Throughout this section, we assume that $G$ is of arbitrary type, even though the results of~\cite{Lec} apply only when $G$ is of simply laced ($ADE$) type while the results of~\cite{Ing} apply only when $G$ is of type $A$. We also fix $v\leq w\in W$ and $(\bv,\bw)\in\Red(v,w)$ with $\bw=(i_1,i_2,\dots,i_m)$.

\subsection{Cluster algebras}
We refer the reader to~\cite{FZ,FWZ_book} for a general introduction to cluster algebras. A \emph{quiver} is a directed graph $Q$ without directed cycles of size~$1$ and~$2$. An \emph{ice quiver} is a quiver $Q$ together with a partition of its vertex set into two sets $V(Q)=V_f(Q)\sqcup V_m(Q)$ of \emph{frozen} and \emph{mutable} vertices, respectively. A \emph{seed} is a pair $(Q,\bx)$ where $\bx=(x_j)_{j\in V(Q)}$ is a collection of algebraically independent elements of the \emph{ambient field} $\Fcal=\C(y_1,\dots,y_r)$ of rational functions in some number of variables. For a mutable vertex $v\in V_m(Q)$, the \emph{mutation} $\mu_v$ is an operation that changes a seed $(Q,\bx)$ into another seed $(Q',\bx')$. The elements of all seeds that can be obtained from $(Q,\bx)$ via a sequence of mutations are called \emph{cluster variables}, where $(x_j)_{j\in V_f(Q)}$ are referred to as \emph{frozen variables}. 

The \emph{cluster algebra} $\Acal(Q,\bx)$ is the subring of $\Fcal$ generated by all cluster variables, frozen variables, and inverses of frozen variables. A \emph{cluster algebra structure} on a variety $X$ is a ring isomorphism $\phi:\Acal(Q,\bx)\xrasim \C[X]$ from some cluster algebra to the ring of regular functions on $X$. Under this isomorphism, each cluster variable becomes a regular function on $X$. Clearly, the isomorphism $\phi$ is determined by its values on the elements of the initial seed $\bx$. 

As we mentioned in the introduction, cluster structures have been previously constructed on special families of open Richardson varieties, namely, on unipotent cells, double Bruhat cells in type $A$, and open positroid varieties; see~\cite{FZ3,Scott,GLS_quantum,SSBW,GL_cluster,GoYa}.

\subsection{Leclerc's construction}\label{ssec:Lec}
For each $r\in[m]$, recall that we have a regular function $\fR_r:N\to \C$ given by~\eqref{eq:fR_dfn}. It is not in general irreducible as an element of the coordinate ring $\C[N]$. When $G$ is of simply laced type, by~\cite[Corollary~4.4]{Lec}, there are exactly $\ell(w)-\ell(v)$ irreducible factors of $\prod_{r\in[m]} \fR_r$, and they form the initial seed in Leclerc's conjectural cluster algebra structure on $\Rich_v^w$. The following conjecture applies to $G$ of arbitrary type, but is open even when $G$ is simply laced.
\begin{conjecture}\label{conj:irr}
There exists a unique family $(F_r)_{r\in\Jo}$ of irreducible elements of $\C[N]$ such that for each $r\in[m]$, we have
\begin{equation*}%
  \fR_r=\prod_{j\in\Jo:\,j\geq r} F_j^{p_{r,j}}
\end{equation*}
for some nonnegative integers $p_{r,j}\geq0$ satisfying $p_{r,r}=1$ for all $r\in\Jo$.
\end{conjecture}

So far, the functions $\fR_r$ are defined on $N$. We explain how to convert them into functions on the flag variety. Leclerc works with the left-sided quotient $\BG$. In order to view each $\fR_r$ as a function on $\LRich_v^w$, he uses the isomorphism~\eqref{eq:NvN_to_Rich} between $\LRich_v^w$  and $\capBWB(N\dv \cap \dv N)$. Specifically, each $\fR_r$ is a function on $\LRich_v^w$ sending $B_-\dv h\mapsto \fR_r(h)$ for $h\in \dv^{-1} \capBWB(N\dv \cap \dv N)\subset N$. By \cref{lemma:Lec_NN'}, we have $\fR_r(yh)=\fR_r(h)$ for all $y\in N\cap \dv^{-1}N_-\dv$. By \cref{lemma:factor_NN'}, the multiplication map gives an isomorphism
\begin{equation*}%
  (N\cap \dv^{-1}N_-\dv) \times \dv^{-1} \capBWB(N\dv \cap \dv N) \xrasim \dv^{-1} \capBWB(\dv N).
\end{equation*}
(Here we have used the fact that left multiplication by $N\cap \dv^{-1}N_-\dv$ preserves the subset $\dv^{-1} B_-wB_-$.) 
Thus $\fR_r$ sends $B_-\dv h\mapsto \fR_r(h)$ where $h$ is now any element of $\dv^{-1} \capBWB(\dv N)$. Recall from \cref{prop:pre_pre_twist} that the element $\yR:=[g^T\dw]_L^+$  satisfies $\yR\in \dv^{-1} \capBWB(\dv N)$. We see that the right pre-twist isomorphism from \cref{dfn:pre_twist} is exactly the missing ingredient needed to define Leclerc's functions on the right-sided quotient $G/B_-$.
\begin{definition}\label{dfn:Lec_fR}
We consider Leclerc's minors $\fR_r$ as regular functions on $\Rich_v^w$, sending $gB_-\mapsto \fR_r(\yR)$, where $g\in N\dv$ and $\yR:=[g^T\dw]_L^+$.
\end{definition}

When $G$ is of simply laced type, Leclerc~\cite{Lec} defined a quiver $\QLec$ on $\ell(w)-\ell(v)$ vertices in terms of irreducible morphisms of preprojective algebra modules. He showed that the resulting seed gives rise to a cluster subalgebra of $\C[\Rich_v^w]$ and conjectured that this subalgebra coincides with $\C[\Rich_v^w]$. Assuming \cref{conj:irr}, the vertex set of $\QLec$ is naturally identified with $\Jo$.

\subsection{Ingermanson's construction}
We continue to assume that $G$ is of arbitrary type. The results of~\cite{Ing} apply when $G=\SL_n(\C)$.  However, the constructions below can be extended to the case of arbitrary $G$. 
 We note that~\cite{Ing} works inside $G/B$ rather than $G/B_-$, and thus make the necessary modifications (such as applying the involution $x\mapsto x^\theta$ from \cref{sec:inv}) to match her notation to our notation.

Following~\cite[Proposition~V.3]{Ing}, let us consider a labeling of the chambers of the wiring diagram for $\bw$ by minors which are very closely related to the functions $\fL_r$ defined in~\eqref{eq:fL_dfn}. Specifically, set 
\begin{equation}\label{eq:fLp_dfn}
  \fLp_r:=\Delta^{\wi{r-1} w_0\om_{i_r^\ast}}_{\vi{r-1} w_0\om_{i_r^\ast}}: B \to \C.
\end{equation}
When $G=\SL_n(\C)$, we therefore have
\begin{equation}\label{eq:fR_fL_fLp_chmnr}
  \fR_r=\chmnrR_{\RegL_r}, \quad \fL_r=\chmnrL_{\RegR_r}, \quad\text{and}\quad \fLp_r=\chmnrL_{\RegL_r};
\end{equation}
cf.~\eqref{eq:fR_fL_chmnr}, \cref{sec:A_example}, and \cref{ex:gracie}. 

Following~\cite[Notation~III.28]{Ing}, we choose a representative for $gB_-$ which differs from the $z$-matrix in \cref{thm:MR_chamber_ans} by a torus element; cf.~\eqref{eq:zing=z}.
\begin{definition}\label{dfn:zing}
Let $gB_-\in \Rich_v^w$. Let $\zing\in B_-$ be the unique matrix satisfying $\zing wB_-=gB_-$ and $\Delta^{w w_0\om_{i^\ast}}_{v w_0\om_{i^\ast}}(\zing^T)=1$ for all $i\in I$. Each $\fLp_r$ is considered as a regular function on $\Rich_v^w$ sending $gB_-\mapsto \fLp_r(\zing)$.
\end{definition}
\begin{remark}
In terms of the usual matrix minors when $G=\SL_n(\C)$, we have $\Delta^{w w_0\om_{i^\ast}}_{v w_0\om_{i^\ast}}=\Delta^{w [i+1,n]}_{v[i+1,n]}$ and $\fLp_r=\Delta^{\wi{r-1} [i_r+1,n]}_{\vi{r-1} [i_r+1,n]}$.
\end{remark}
 The cluster variables $(X_j)_{j\in\Jo}$ introduced in~\cite[Proposition~V.1]{Ing} are irreducible factors of these regular functions. There are exactly $\ell(w)-\ell(v)$ of them, and~\cite[Section~4]{Ing} describes a combinatorial procedure for how each function $\fLp_r$ factorizes into irreducibles. In particular, these functions are labeled by the elements of~$\Jo$ and \cref{conj:irr} becomes manifestly true in this description. 

\begin{figure}
  \includegraphics[width=1.0\textwidth]{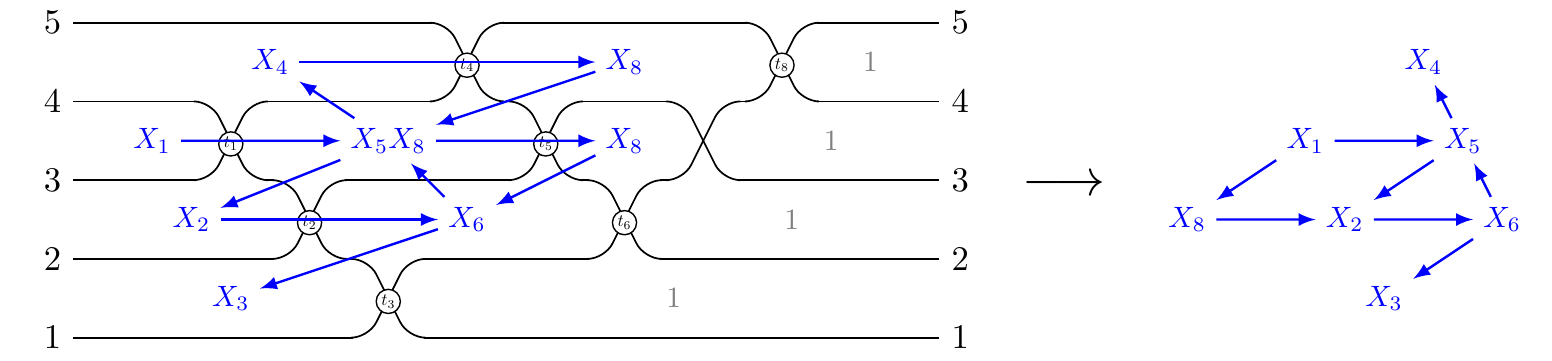}
  \caption{\label{fig:gracie_arrows} Constructing the quiver of~\cite{Ing} from the BFZ quiver by combining the arrows between the irreducible factors; cf.~\cite[Figure~7.8]{Ing}.}
\end{figure}

\cite[Definition~VII.2]{Ing} gives an entirely different construction of the quiver $\QIng$: it is simply obtained from the \emph{BFZ quiver}~\cite[Definition~2.2]{FZ3} associated to the wiring diagram of $\bw$ by summing up the contributions of all arrows between each pair of irreducible factors and removing all loops and $2$-cycles; see \cref{fig:gracie_arrows} and \cref{ex:gracie}. The following conjecture is supported by vast computational evidence.
\begin{conjecture}\label{conj:cluster} Let $G=\SL_n(\C)$.
\begin{enumerate}
\item The quivers $\QLec$ and $\QIng$ are isomorphic.
\item Assuming \cref{conj:irr}, the isomorphism is given by naturally identifying both vertex sets with $\Jo$. 
\item For each $v\leq w\in W$, these quivers are locally acyclic in the sense of~\cite{Muller}.
\end{enumerate}
\end{conjecture}
\noindent 
We shall return to the study of the cluster structure of Richardson and related varieties in joint work with Melissa Sherman-Bennett and David Speyer.\footnote{We understand that Khrystyna Serhiyenko and Melissa Sherman-Bennett have a forthcoming work that verifies some parts of the Conjecture~\ref{conj:cluster}.}

\subsection{Relating the two cluster structures}

\begin{figure}
\begin{tabular}{c}%
\includegraphics[width=0.98\textwidth]{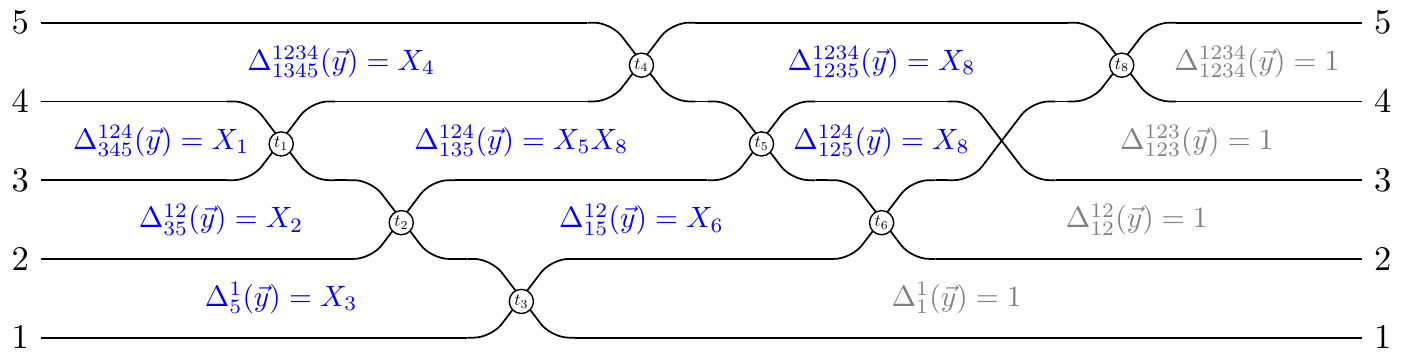}%
\\
\\
\includegraphics[width=0.98\textwidth]{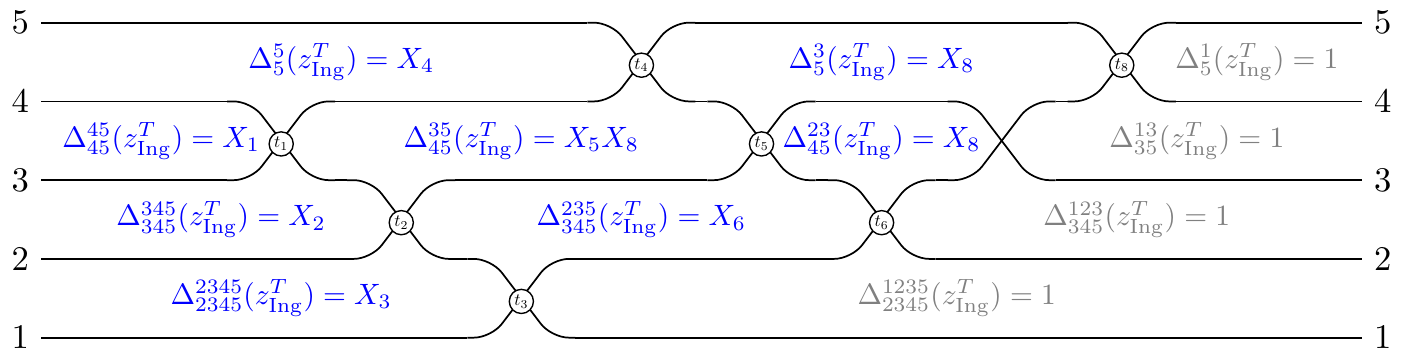}%
\end{tabular}
  \caption{\label{fig:gracie_compare} Comparing the conjectural cluster structures of~\cite{Lec} and of~\cite{Ing}; cf. \cref{thm:compare,ex:gracie}.}
\end{figure}

\begin{theorem}\label{thm:compare}
The functions $\fR_r$ and $\fLp_r$ coincide as regular functions on $\Rich_v^w$. More precisely, in the notation of \cref{dfn:Lec_fR,dfn:zing}, we have 
\begin{equation}\label{eq:flp=fr}
  \fLp_r(\zing^T)=\fR_r(\yR) \quad\text{for each $r\in[m]$.}
\end{equation}
\end{theorem}

\begin{example}\label{ex:gracie}
Let us consider the running example of~\cite{Ing}. We have $n=5$, $v=s_3$, $\bw=(3,2,1,4,3,2,3,4)$, $\Jo=\{1,2,3,4,5,6,8\}$. The MR-parameters $(t_j)_{j\in\Jo}$ are expressed in terms of the cluster variables $(X_j)_{j\in \Jo}$ by applying the Chamber Ansatz~\eqref{eq:A_ch_ans} to the chamber labels in \cref{fig:gracie_arrows}:
\begin{equation*}%
  t_1=\frac{X_2X_4}{X_1X_5X_8}, \quad t_2=\frac{X_3X_5X_8}{X_2X_6},\quad t_3=\frac{X_6}{X_3},\quad t_4=\frac{X_5}{X_4},\quad t_5=\frac{X_6}{X_5X_8},\quad t_6=\frac{X_8}{X_6},\quad t_8=\frac1{X_8}.
\end{equation*}
With this substitution, we compute

\noindent\makebox[1.0\textwidth]{\scalebox{0.94}{$\displaystyle \yR=\smat{
1 & \frac{X_{3} X_{4} X_{5} + X_{2} X_{4} + X_{1} X_{6}}{X_{4} X_{5} X_{6}} & \frac{X_{3} X_{5} + X_{2}}{X_{6}} & \frac{{\left(X_{1} + X_{2}\right)} X_{3}}{X_{2} X_{8}} & X_{3} \\
0 & 1 & X_{5} & \frac{{\left(X_{1} + X_{2}\right)} X_{6}}{X_{2} X_{8}} & X_{6} \\
0 & 0 & 1 & \frac{X_{2} X_{4} + X_{1} X_{6}}{X_{2} X_{5} X_{8}} & 0 \\
0 & 0 & 0 & 1 & X_{8} \\
0 & 0 & 0 & 0 & 1
}, \quad \zing^T=\smat{
\frac{1}{X_{3}} & \frac{X_{3} X_{5} + X_{2}}{X_{2} X_{6}} & \frac{X_{1} + X_{2}}{X_{1} X_{8}} & \frac{X_{5}}{X_{4}} & 1 \\
0 & \frac{X_{3}}{X_{2}} & 0 & -\frac{X_{1} X_{6}}{X_{2} X_{4}} & -\frac{X_{2} X_{4} + X_{1} X_{6}}{X_{2} X_{5}} \\
0 & 0 & \frac{X_{2}}{X_{1}} & \frac{X_{5} X_{8}}{X_{4}} & X_{8} \\
0 & 0 & 0 & \frac{X_{1}}{X_{4}} & 0 \\
0 & 0 & 0 & 0 & X_{4}
}.$}}
One can check that the matrix $\zing^\th$ indeed coincides with the matrix given in~\cite[Example~VII.27]{Ing}. The minors of $\yR$ and $\zing^T$ appearing in \cref{thm:compare} are compared in \cref{fig:gracie_compare}.
\end{example}

\begin{proof}[Proof of \cref{thm:compare}]
Let $(\Yp,\Yo,\Ym)\in N\times \H\times N_-$ be given by~\eqref{eq:YpYoYm_dfn}. By~\eqref{eq:yR=Yp}, we have $\Yp=\yR$. Let $z:=(\dw \Ym\dw^{-1})^T$ be given by~\eqref{eq:z_wyw}. As explained after~\eqref{eq:z_wyw}, we have $z\in N_-$ and $zwB_-=gB_-$.

Our first goal is to show that $z$ differs from $\zing$ by a torus element. Specifically, we claim that 
\begin{equation}\label{eq:zing=z}
  \zing:=(\dw \Yo\Ym \dw^{-1})^T=z\cdot \dw\Yo\dw^{-1}
\end{equation}
\noindent satisfies the conditions of \cref{dfn:zing}. 
First, we have $z\in N_-$ and $\dw\Yo\dw^{-1}\in\H$, so $\zing\in B_-$. Second, we have $\zing wB_-=z \dw\Yo B_-=z w B_-= gB_-$. It remains to show that $\Delta^{w w_0\om_{i^\ast}}_{v w_0\om_{i^\ast}}(\zing^T)=1$ for all $i\in I$.  By \cref{eq:dom_H}, we have
\begin{equation*}%
  \Delta^{w w_0\om_{i^\ast}}_{v w_0\om_{i^\ast}}(\zing^T)= 
\Delta^{w w_0\om_{i^\ast}}_{v w_0\om_{i^\ast}}(\dw\Yo\dw^{-1}\cdot z^T)=
\Delta^{w w_0\om_{i^\ast}}_{w w_0\om_{i^\ast}}(\dw\Yo\dw^{-1})\cdot 
\Delta^{w w_0\om_{i^\ast}}_{v w_0\om_{i^\ast}}(z^T).
\end{equation*}
We compute each of the two factors separately.  Applying~\eqref{eq:dom_iota} and then~\eqref{eq:dom_dfn}, we get
\begin{equation*}%
  \Delta^{w w_0\om_{i^\ast}}_{w w_0\om_{i^\ast}}(\dw\Yo\dw^{-1}) = 
\Delta^{w\om_{i}}_{w \om_{i}}((\dw\Yo\dw^{-1})^\iota) =
\Delta^{w\om_{i}}_{w \om_{i}}(\line{w}\Yo^{-1} \lline{w^{-1}}) =
\Delta^{\om_{i}}_{\om_{i}}(\lline{w^{-1}}\line{w}\Yo^{-1} \lline{w^{-1}}\line{w}) =
\Delta^{\om_{i}}_{\om_{i}}(\Yo^{-1}).
\end{equation*}
It suffices to assume that $g=\gt$ is MR-parametrized for $\bt\in(\R_{>0})^{\Jo}$, since $\Rtp_v^w$ is Zariski dense inside $\Rich_v^w$. Then the first factor $\Delta^{w w_0\om_{i^\ast}}_{w w_0\om_{i^\ast}}(\dw\Yo\dw^{-1})$ is given by the inverse of the right hand side of~\eqref{eq:dom_Yo}. The second factor $\Delta^{w w_0\om_{i^\ast}}_{v w_0\om_{i^\ast}}(z^T)$ coincides with $\fL_j(z^T)$ where $j$ is the maximal index such that $i_j=i$. Applying~\eqref{eq:MR_chamber_ans_inv}, we indeed see that $\Delta^{w w_0\om_{i^\ast}}_{v w_0\om_{i^\ast}}(\zing^T)=1$. We have shown~\eqref{eq:zing=z}.

The proof of~\eqref{eq:flp=fr} proceeds by an entirely similar argument. We have
\begin{equation*}%
  \fLp_r(\zing^T)= 
\Delta^{\wi{r-1} w_0\om_{i_r^\ast}}_{\vi{r-1} w_0\om_{i_r^\ast}}(\dw\Yo\dw^{-1}\cdot z^T)=
\Delta^{\wi{r-1} w_0\om_{i_r^\ast}}_{\wi{r-1} w_0\om_{i_r^\ast}}(\dw\Yo\dw^{-1})\cdot 
\Delta^{\wi{r-1} w_0\om_{i_r^\ast}}_{\vi{r-1} w_0\om_{i_r^\ast}}(z^T).
\end{equation*}
Restricting to the case $gB_-\in\Rtp_v^w$ with $g=\gt$ as above, the first factor is computed using \cref{lemma:Yo_tnn} and the second factor is computed using~\eqref{eq:MR_chamber_ans_inv}. Computing the right-hand side of~\eqref{eq:flp=fr} using~\eqref{eq:R_chamber_ans_inv}, we deduce the result.
\end{proof}

\appendix
\section{Fixed points of the twist}\label{sec:fixed}
The twist map $\twist_v^w$ is an automorphism of $\Rich_v^w$ which preserves the totally positive part $\Rtp_v^w$ and generalizes the right twist of~\cite{MuSp}. In~\cite[Section~4.5]{KarpCS}, Karp posed the problem of finding the fixed points of the right twist map on open positroid varieties and their totally positive parts. This section is motivated by a related question for open Richardson varieties.
\begin{question}
Does $\twist_v^w:\Rtp_v^w\xrasim\Rtp_v^w$ have a unique fixed point for all $v\leq w\in W$? 
\end{question}
We show below that in general, the answer is ``no'' (even for open positroid varieties). We start with an interesting example where the answer appears to be ``yes.''

\begin{conjecture}
Let $G=\SL_n(\C)$, $v=\id$, and $w=w_0$. Then the unique fixed point of $\twist_v^w$ inside $\Rtp_v^w$ equals $gB_-$ where $g=(g_{i,j})_{i,j=1}^n$ is an upper triangular matrix given by
\begin{equation*}%
  g_{i,j}=\sqrt{{n-i\choose j-i}{j-1\choose j-i}},\quad\text{for $1\leq i\leq j\leq n$.}
\end{equation*}
\end{conjecture}
\begin{example}
For $n=5$, the above matrix $g$ is given by
\begin{equation*}%
g=  \smat{%
    \sqrt1 & \sqrt4 & \sqrt6 & \sqrt4 & \sqrt1\\
    0 & \sqrt1 & \sqrt6 & \sqrt9 & \sqrt4\\
    0 & 0 & \sqrt1 & \sqrt6 & \sqrt6\\
    0 & 0 & 0 & \sqrt1 & \sqrt4\\
    0 & 0 & 0 & 0 & \sqrt1
  }.%
\end{equation*}
\end{example}

We now give an example where the twist $\twist_v^w$ has no fixed points on $\Rtp_v^w$. 
\begin{example}
Consider the open positroid variety $\Pio_v^w$ from~\cite[Section~A.2]{MuSp}. Its Zariski closure is the subvariety of $\Gr(4,8)$ given by the vanishing of four Pl\"ucker coordinates
\begin{equation*}%
  \Delta_{1234}=\Delta_{3456}=\Delta_{5678}=\Delta_{1278}=0.
\end{equation*}
The frozen variables are given by
\begin{equation*}%
  \Delta_{1238},\Delta_{2348},\Delta_{2345},\Delta_{2456},\Delta_{4567},\Delta_{4678},\Delta_{1678},\Delta_{1268}.
\end{equation*}
The right twist inverts the frozen variables by~\cite[Equation~(9)]{MuSp}. Thus, in order for them to stay preserved (and be positive real), all frozen variables have to be equal to $1$. There are also five mutable cluster variables
\begin{equation*}%
  \Delta_{1248},\Delta_{2346},\Delta_{4568},\Delta_{2678},\Delta_{2468}.
\end{equation*}
Using~\cite[Equations~(A.1)--(A.2)]{MuSp}, one can write down a system of equations on these five variables in order for them to stay preserved under the twist map, and check that it admits no positive real solutions.
\end{example}
\begin{remark}
The above example is also an example of the twist map having infinite order on the torus quotient $\Pio_v^w/\H$. This suggests an interesting interplay between the behavior of totally positive fixed points of the twist map and its periodicity properties; see~\cite{GP_Zam} for related results.
\end{remark}

\bibliographystyle{alpha_tweaked}
\bibliography{twist}
\end{document}